\documentclass[11pt,reqno]{amsart}     
\usepackage{graphics,array}
\usepackage{amsmath,amsthm}
\usepackage{amsfonts}
\usepackage{amssymb}
\usepackage{hyperref}
\usepackage{color}
\newcounter{foo}
 \newtheorem{theorem}{Theorem}[section]
 \newtheorem{corollary}[theorem]{Corollary}
\newtheorem{conjecture}[theorem]{Conjecture}
 \newtheorem{lemma}[theorem]{Lemma}
 
 \newtheorem{example}[theorem]{Example}

 \newtheorem{characterization}[foo]{Characterization}
 \theoremstyle{remark}
 
 \theoremstyle{definition}
 \newtheorem{definition}[theorem]{Definition}

\newcommand{\citep}[1]{\cite{#1}}

\DeclareMathAlphabet{\mathpzc}{OT1}{pzc}{m}{it}
\DeclareMathOperator{\minanz}{minanz}
\def\mathpzc{\mathsf}

\newcommand{\GETOUT}[1]{}
\newcommand{\diag}{\mathsf{diag}}
\newcommand{\Parking}{\mathsf{Park}}
\newcommand{\decomp}{\mathsf{decomp}}  
\newcommand{\inc}{\mathsf{inc}}
\newcommand{\sperm}{\mathsf{smperm}}
\newcommand{\CanonTop}{\mathsf{CanonTop}}
\newcommand{\canontop}{\mathsf{canontop}}
\newcommand{\Bounce}{\mathsf{Bounce}}
\newcommand{\hBounce}{\mathsf{HagBounce}}
\newcommand{\hbounce}{\mathsf{hagbounce}}
\newcommand{\bounce}{\mathsf{bounce}}
\newcommand{\ebounce}{\mathsf{ebounce}}
\newcommand{\obounce}{\mathsf{obounce}}
\newcommand{\bounceWeight}{\mathsf{parabounce}}
\newcommand{\type}{\mathsf{type}}
\newcommand{\BiComp}{\mathsf{BiComp}}
\newcommand{\mtri}{\includegraphics{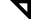}}
\newcommand{\UTBiComp}{\mathsf{BiComp}^{\mtri}}
\newcommand{\Rec}{\mathsf{Rec}}
\newcommand{\RecInc}{\mathsf{Rec}^{\star}}
\newcommand{\RecDec}{\mathsf{Rec}^{\dagger}}
\newcommand{\Stable}{\mathsf{Stable}}
\newcommand{\SqRec}{\mathsf{SqRec}}
\newcommand{\THSqRec}{{\mathsf{SqRec}}^{\rightarrow}}
\newcommand{\upperp}{\mathsf{upper}}
\newcommand{\lowerp}{\mathsf{lower}}

\newcommand{\Mat}[1]{\left(\begin{smallmatrix} #1 \end{smallmatrix}\right)}
\newcommand\nn{\mathsf{n}}
\newcommand\sss{\mathsf{s}}
\newcommand\ee{\mathsf{e}}
\newcommand\ww{\mathsf{w}}
\newcommand{\R}{\mathbb{R}}
\newcommand{\PP}{\mathcal{P}}
\newcommand\Para{\mathpzc{Para}}
\newcommand\Ribbon{\mathpzc{Ribbon}}
\newcommand\SetPara{\mathpzc{SetPara}}
\newcommand\Dyck{\mathpzc{Dyck}}
\newcommand{\BSP}{\mathpzc{LowerPara}}
\newcommand{\dyck}{\mathsf{dyck}}
\newcommand{\Myshape}[4]{\Lambda_{#3 , #4}(#1,#2)}
\newcommand\ttm{\phi}       
\newcommand{\RtoSP}{\Omega} 
\newcommand{\mtp}{\mathsf{mtp}}
\newcommand{\Posets}{\mathsf{Posets}}
\newcommand\area{\mathsf{area}}
\newcommand\uarea{\mathsf{uarea}}
\newcommand\larea{\mathsf{larea}}
\newcommand{\level}{\mathsf{level}}

\newcommand{\sgdistribution}{F}
\newcommand{\sgdistributionHaglund}{H}
\newcommand{\sgdistributionSticked}{S}
\newcommand{\diagonalLengths}{\mathsf{diaglen}}
\newcommand{\wave}{\mathsf{wave}}

\begin{document}
 \title[Polyominoes, the sandpile model, and a $q,t$-Narayana polynomial]{Parallelogram polyominoes, the sandpile model on a complete bipartite graph, and a $q,t$-Narayana polynomial}
 \author{Mark Dukes and Yvan Le Borgne}
 \address{Department of Computer and Information Sciences, University of Strathclyde, Glasgow G1 1XH, United Kingdom.}
 \address{CNRS, LaBRI, Universit\'e Bordeaux 1, 351 cours de la Lib\'eration, 33405 Talence cedex, France.}

\begin{abstract}
We classify recurrent configurations of the sandpile model on the complete bipartite graph $K_{m,n}$ in which one designated vertex is a sink.
We present a bijection from these recurrent configurations to decorated parallelogram polyominoes whose bounding box is a $m \times n$ rectangle. 
Several special types of recurrent configurations and their properties via this bijection are examined. 
For example, recurrent configurations whose sum of heights is minimal are shown to correspond to polyominoes of least area.
Two other classes of recurrent configurations are shown to be related to bicomposition matrices, a matrix analogue of set partitions,
and (2+2)-free partially ordered sets.

A canonical toppling process for recurrent configurations gives rise to a path within the associated parallelogram polyominoes.
This path bounces off the external edges of the polyomino, and is reminiscent of Haglund's well-known bounce statistic for Dyck paths.
We define a collection of polynomials that we call $q,t$-Narayana polynomials, 
defined to be the generating function of the bistatistic $(\area,\bounceWeight)$ on the set of parallelogram polyominoes, 
akin to the $(\area,\hbounce)$ bistatistic defined on Dyck paths in Haglund (2003). 
In doing so, we have extended a bistatistic of Egge et al. (2003) 
to the set of parallelogram polyominoes.
This is one answer to their question concerning extensions to other combinatorial objects.

We conjecture the $q,t$-Narayana polynomials to be symmetric and prove this conjecture for numerous special cases.
We also show a relationship between Haglund's $(\area,\hbounce)$ statistic on Dyck paths, and our bistatistic 
$(\area,\bounceWeight)$ on a sub-collection of those parallelogram polyominoes living in a $(n+1)\times n$ rectangle.
\end{abstract}
\maketitle

\section{Introduction}
The abelian sandpile model \citep{Dhar:2006} is a discrete diffusion model whose states are distributions of grains on the vertices of a general directed graph.
A vertex of a graph is {\it{stable}} if the number of grains at the vertex is strictly smaller than its out-degree, otherwise it is called {\it{unstable}}.
In addition to a randomised addition of grains, 
the dynamics of this model requires that an unstable vertex donate a grain to each of its direct successors. This is the so-called toppling or avalanche process of the model.
This process defines a Markov chain on stable states. 
In the particular case of graphs, there exist many bijections from recurrent states to spanning trees 
of the same graph ~\citep{Bernardi:2008,ChebikinPylyavskyy:2005,CoriLeBorgne:2003,Dhar:2006}.
Another rich result concerning the model shows that the recurrent states, together with a certain binary operation, form an abelian group called the sandpile group (see e.g. ~\cite{CoriRossin:2000}).

In this paper we classify recurrent configurations of the sandpile model on the complete bipartite graph $K_{m,n}$ in which a single designated vertex is a sink.
Cori and Poulalhon~\citep{CoriPoulalhon:2002} classified recurrent configurations of the complete $(k+1)$-partite graph 
$K_{p_1,\ldots,p_k,1}$ where the sink is the single vertex that is connected to all other vertices.
Their paper showed that such recurrent configurations can be related to a generalisation of parking functions, 
and also exhibited a rich connection to a {\L}ukasiewicz language which aided in the enumeration of their new generalised parking functions.

This paper complements their paper by breaking the symmetry of $K_{p,q,1}$ and showing 
how recurrent configurations of $K_{m,n}$ that have one sink can be interpreted as parallelogram (or staircase) polyominoes.
We study several different types of recurrent configurations, which are illustrated in Figure~\ref{SPECIAL:TYPES}.
We show that minimal recurrent configurations (those whose sum of heights is minimal) correspond 
to ribbon parallelogram polyominoes (those having minimal area). 
Almost non-zero configurations are configurations in which one distinguished vertex is allowed to be empty, and all others are non-empty.

Bicomposition matrices are defined to be square matrices of sets whose non-empty entries partition the set $\{1,\ldots,n\}$ 
and for which there are no rows or columns consisting of only empty sets.
We show that minimal almost non-zero (`minanz') recurrent configurations on $K_{n,n}$ are in one-to-one correspondence with bicomposition matrices.
A special class of these minanz recurrent configurations are in one-to-one correspondence with upper triangular bicomposition matrices.
Such matrices were the subject of several recent papers that presented surprising connections between five seemingly disparate combinatorial objects: (2+2)-free posets, Stoimenow matchings,
ascent sequences, permutations avoiding a length-3 bivincular pattern, and a class of upper triangular matrices ~\cite{BousquetMelouUs:2010,DukesParviainen:2010}. 
We present the composition of this new correspondence and the correspondence between matrices and posets to show how the heights of a configuration can be read from the corresponding poset.
It was through these structures that we first noticed a connection to the sandpile model and the initial motivation behind this paper.

During the last two decades, a series of papers have examined a power series $C_n(q,t)$ 
that has become known as the $q,t$-Catalan function (or polynomial).
This power series was introduced by Garsia and Haiman~\cite{GarsiaHaiman:1996} and has important links to algebraic geometry and representation theory. 
In the original paper they showed that two special cases of this polynomial had combinatorial significance. 
The first was that $C_n(q,1)$ was the generating function of the area statistic over Dyck paths having semi-length $n$. 
The second was that $q^{n \choose 2}C_n(q,1/q)$ is the $n$th $q$-Catalan number.

Haglund introduced a new statistic `bounce' (we will call this $\hbounce$ so as not to confuse it with another statistic) and conjectured that 
$C_n(q,t)$ was the generating function of the bistatistic $(\area,\hbounce)$ on the set of all Dyck path of semi-length $n$.
Garsia and Haglund~\cite{GarsiaHaglund:2001,GarsiaHaglund:2002} proved this conjecture using methods from the theory of symmetric functions.
(See Haglund~\citep{Haglund:2008} for a concise overview of these results.)
Egge, Haglund, Kremer and Killpatrick~\citep{EggeHaglundKremerKillpatrick:2003} asked if the lattice path statistics for $C_n(q,t)$ can be extended, 
in a way which preserves the rich combinatorial structure, to related combinatorial objects. 

We present a pair of statistics $(\area,\bounceWeight)$ on parallelogram polyominoes which is one answer to their question.
We call the resulting polynomials $q,t$-Narayana polynomials since they specialise to the Narayana numbers for the case $q=t=1$.
The bivariate generating function for this pair of statistics $F_{m,n}(q,t)$  appears to be symmetric in both $q,t$ and $m,n$.
We conjecture and discuss the symmetry of this bistatistic and prove it for numerous special cases.

The outline of the paper is as follows:
Section \ref{SEC:2} defines parallelogram polyominoes and some notation related to these objects.
Section \ref{SEC:3} attends to the classification of recurrent configurations and a proves a bijection from this set to a set of decorated parallelogram polyominoes.
Section \ref{SEC:4} looks at several different types of (minimal) recurrent configurations and their relationship to other structures.
The different collections we look at are illustrated and summarised in Figure~\ref{SPECIAL:TYPES}.
In Section \ref{SEC:5} we define the $q,t$-Narayana polynomials and investigate their symmetry.

\section{Parallelogram polyominoes -- background and definitions}
\label{SEC:2}
First we will define parallelogram polyominoes and some aspects thereof.
A {\it{parallelogram polyomino}} is a polyomino such that the intersection with every line of slope $-1$ is a connected segment. 
Parallelogram polyominoes may be described in several ways and in this paper we will see three distinct ways of doing this with each serving a different purpose.
Let us write $\Para_{m,n}$ for the collection of parallelogram polyominoes whose bounding rectangle is $[0,m] \times [0,n]$.
These polyominoes have been studied in several papers, see for example \citep{BousquetMelou:1996,LerouxRassart:2001}.
Every polyomino $\PP \in \Para_{m,n}$ can be uniquely described by a pair 
$(\upperp(\PP),\lowerp(\PP))$ of paths which begin at $(0,0)$, end at $(m,n)$, 
take steps in the set $\{\nn = (0,1),\ee=(1,0)\}$, and only touch at their endpoints.

Given $\PP \in \Para_{m,n}$, let $\area(\PP)$ be the area of the finite region enclosed by the upper and lower paths defining $\PP$. 
Let us define $\uarea(\PP)$ and $\larea(\PP)$ to be the areas of the regions above and below $\PP$ in its bounding $m\times n$ rectangle, 
so $\larea(\PP)+\area(\PP)+\uarea(\PP) = mn$.
Let $\Ribbon_{m,n}$ be the set of polyominoes $\PP$ in $\Para_{m,n}$ whose area is $m+n-1$, the minimal value,
which we will call {\it{ribbon polyominoes}}.

\begin{example}\label{exampleoneone}
The first polyomino $\PP_1$ in Figure~\ref{threeexamples} is a parallelogram polyomino. 
It can be described by the pair of paths $(\upperp(\PP_1),\lowerp(\PP_1)) = ( \nn\ee\nn\ee\ee\nn\ee\ee\nn\ee\ee , \ee\ee\ee\nn\ee\nn\ee\ee\nn\ee\nn)$.
The second polyomino $\PP_2$ is not a parallelogram polyomino, as is witnessed by the two disjoint segments the dashed line of slope $-1$ meets.
Polyomino $\PP_3$ is a ribbon (parallelogram) polyomino.
\begin{figure}
\centerline{
	\begin{tabular}{c}
	\includegraphics{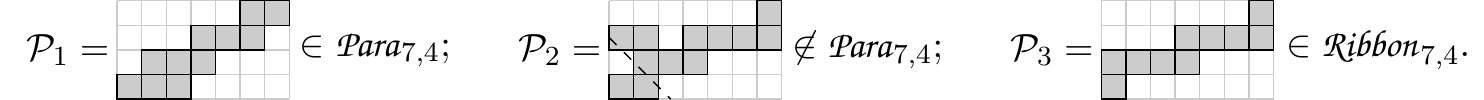}
	\end{tabular}
}
\caption{\label{threeexamples} Three examples.}
\end{figure}
\end{example}

Given $\PP \in \Para_{m,n}$, let $\Bounce(\PP)$ be the unique path from $(m-1,n)$ to $(0,0)$ which is defined as follows: 
Starting from $(m-1,n)$, the path goes south until it encounters a vertex on the lower path of $\PP$. 
The path then turns to the west and continues straight until it encounters a vertex on the upper path of $\PP$.
The path turns again to the south until it encounters a vertex on the lower path of $\PP$, and so on, until it reaches $(0,0)$.

Let $\bounce(\PP)=(c_1,c_2,\ldots)$ be the sequence of numbers where $c_1$ is the number of initial south $\sss$ steps in 
$\Bounce(\PP)$, $c_2$ is the number of contiguous west steps that follow the initial run of south steps in $\Bounce(\PP)$, and so forth.

\begin{example} \label{bpexamp}
The bounce path is indicated by a thick directed line in the following diagrams:
	\begin{center}
	\begin{tabular}{c}
	\includegraphics{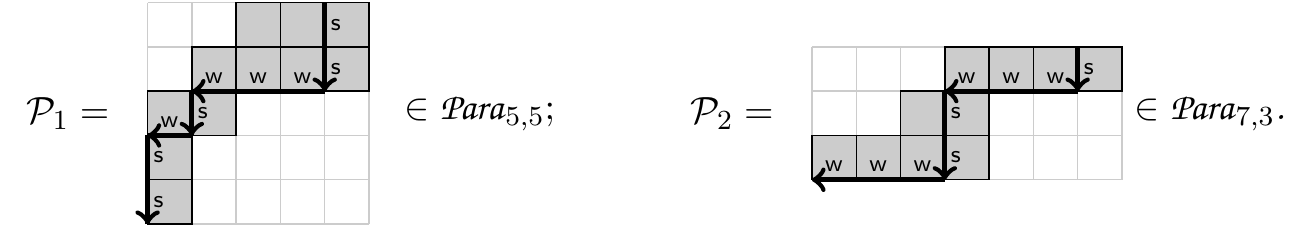}
	\end{tabular}
	\end{center}
For $\PP_1$ we have
$\Bounce(\PP_1)=(\sss,\sss,\ww,\ww,\ww,\sss,\ww,\sss,\sss)$ and $\bounce(\PP_1)=(2,3,1,1,2)$.
For $\PP_2$ we have 
$\Bounce(\PP_2)=(\sss,\ww,\ww,\ww,\sss,\sss,\ww,\ww,\ww)$ and $\bounce(\PP_2)=(1,3,2,3)$.
\end{example}
In the appendix we give two characterisations of parallelogram polyominoes.
Characterisation~\ref{bpcharact} is in terms of the bounce path of a polyomino. 
Characterisation~\ref{seqcharact} is in terms of the two integer sequences from which a polyomino may be constructed.

\section{The sandpile model on $D_{m,n}$ and recurrent configurations}
\label{SEC:3}
In this section we show how recurrent configurations of the sandpile model~\citep{Dhar:2006} 
on $K_{m,n}$ which has a designated vertex that acts as a sink can be classified in terms of parallelogram polyominoes. 
In what follows we will refer to the directed bipartite graph $K_{m,n}$ that has a designated vertex, which we will call the {\it{sink}}, as $D_{m,n}$.
The purpose in doing so is to avoid confusion about two non-equivalent choices for the sink in $K_{m,n}$ (in general if we have many sinks then we merge them into a single one).


For a general (stable) configuration $u$ on $D_{m,n}$ we define a collection of cells $f_{m,n}(u) \subseteq [0,m]\times [0,n] \subseteq \R^2$.
Next we define a canonical toppling process for checking recurrent configurations (using Lemma~\ref{knuthcori}) for the graph $D_{m,n}$ and 
go on to show that a configuration being recurrent is equivalent to $f_{m,n}(u)$ being a parallelogram polyomino in $\Para_{m,n}$.
In addition, we show the canonical toppling process of a recurrent configuration to be intimately linked with the bounce path of the corresponding polyomino $f_{m,n}(u)$.

Let $D_{m,n}$ be the directed graph $D_{m,n}(V,E)$ with vertex set $V=\{v_0,\ldots,v_{m+n-1}\}$ and a directed edge set
\[ E = \{ (v_i,v_j) \, :\, \min (i,j) \in \{0,\ldots,m-1\} \mbox{ and } \max(i,j) \in \{m,\ldots,m+n-1\} \}.\]
This is illustrated in Figure~\ref{the:graph}.
We call $\{v_0,\ldots,v_{m-1}\}$ and $\{v_m,\ldots,v_{m+n-1}\}$
the {\it{top}} vertices and {\it{bottom}} vertices of $D_{m,n}$, respectively. 

\begin{figure}
\centerline{\includegraphics{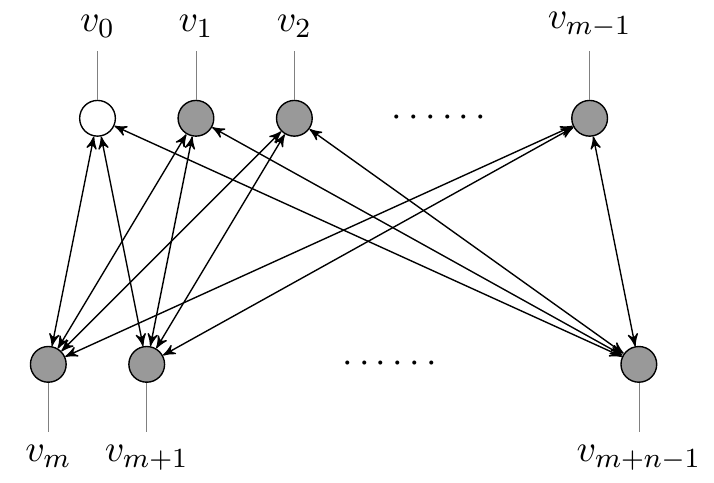}}
\caption{\label{the:graph} The graph $D_{m,n}$.}
\end{figure}

A {\it{configuration}}, or {\it{state}}, $x$ on $D_{m,n}$ is an assignment $$x:\{1,\ldots,m+n-1\} \to \mathbb{N}_0=\mathbb{N} \cup \{0\},$$
the collection of non-negative integers.
The value $x_i = x(i)$ represents the number of grains of sand on vertex $v_i$.
Vertex $v_0$ is treated as the {\it{sink}} and the number of grains of sand on this vertex is generally ignored.
There is, however, one benefit to considering the number of grains at the sink and this will be mentioned in the proof of Theorem~\ref{main:thm}.

Let $e_{ij}=1$ if $(v_i,v_j) \in E$, and 0 otherwise. Let $d_i$ be the total out-degree of vertex $v_i$. 
The toppling aspect of the sandpile model is as follows:
Suppose $x=(x_1,\ldots,x_{m+n-1})$ is a state on $D_{m,n}$.
If $x_i \geq d_i$ for some $i\geq 1$, then let $y=(y_1,\ldots,y_{m+n-1})$ be the state with
$y_i=x_i-d_i$ and set $y_j =  x_j+e_{ij}$ for all $j\neq i$. 
This is denoted $x \stackrel{i}{\rightarrow} y$ and means that in configuration $x$ the unstable vertex $v_i$ topples leading to the configuration $y$.

One continues this toppling procedure until $x_i<d_i$ for all $i \in \{1,\ldots,m+n-1\}$. 
(The order of the topplings does not matter, see e.g. Dhar~\cite{Dhar:2006}.) 
Such a state $(x_1,\ldots,x_{m+n-1})$ is called {\it{stable}}.
Let $\Stable(D_{m,n})$ be the set of stable states on $D_{m,n}$. 
For example, it is easy to see that $|\Stable(D_{m,n})|=n^{m-1}m^n$.

Let us call a configuration $u \in \Stable(D_{m,n})$ {\it{increasing}} if $u_1\leq \cdots \leq u_{m-1}$ and $u_m\leq \cdots \leq u_{m+n-1}$. 
Every configuration $u \in \Stable(D_{m,n})$ may be written uniquely as a pair $\decomp_{m,n}(u)=(\inc_{m,n}(u),\sperm_{m,n}(u))$,
where $\inc_{m,n}(u)=u'$ is an increasing configuration, $\sperm_{m,n}(u)=\pi$ is the lexicographically smallest permutation such 
that $\pi:\{1,\ldots ,m-1\}\to\{1,\ldots,m-1\}$, $\pi:\{m,\ldots,m+n-1\} \to \{m,\ldots,m+n-1\}$, and $u_i=u'_{\pi(i)}$ for all $i \in \{1,\ldots,m+n-1\}$.

\begin{example}
Consider $u=(5,2,1,2,4,2,1,3,2,2,1) \in \Stable(D_{5,7})$.
We have $\inc_{5,7}(u)=(1,2,2,5,1,1,2,2,2,3,4)$ and 
$\sperm_{5,7}(u) =(4,2,1,3,11,7,5,10,8,9,6).$
\end{example}

Let $\sigma(x)$ be the stable state that results from initial state $x$.
A stable state is {\it{recurrent}} if it is $\sigma(x)$ for some $x$ with $x_i \geq d_i$ for all $i\geq 1$. 
Let $\Rec(D_{m,n})$ be the set of recurrent states on $D_{m,n}$ and let $\RecInc(D_{m,n})$ be the set of increasing recurrent states on $D_{m,n}$.
Recurrent states may be classified for general directed graphs by the following lemma (see \citep{Knuth:2011} or \citep{CoriPoulalhon:2002,Dhar:2006}).

\begin{lemma}\label{knuthcori}
Let $D$ be a directed graph on the vertices $V_0,V_1,\ldots,V_n$ with $e_{ij}$ arcs from $V_i$ to $V_j$, and $e_{ii}=0$. 
Let $D$ be balanced (the in-degree of each vertex equals its out-degree).
Then the state $x=(x_1,\ldots,x_n)$ is recurrent iff $x=\sigma(x+a')$ where $a'=(e_{01},\ldots ,e_{0n})$.
Furthermore, in toppling from $x+a'$ to $x$, every vertex $V_1,\ldots,V_n$ topples exactly once.
\end{lemma}

The above lemma applied to the graph $D_{m,n}$ gives us that a stable configuration 
$u\in\Stable(D_{m,n})$ is recurrent iff the Dhar criterion~\citep{Dhar:2006} holds:
\begin{align}\label{topeqn}
\sigma(u_1,\ldots,u_{m-1},1+u_m,\ldots,1+u_{m+n-1}) &= (u_1,\ldots,u_{m-1},u_m,\ldots,u_{m+n-1}).
\end{align}
The order of topplings is unimportant when checking that the Dhar criterion (\ref{topeqn}) holds.
However, for our purposes it will prove useful to fix a canonical toppling process as follows:
Let $Q_1$ be the set of all unstable vertices in the bottom row of $D_{m,n}$ (as a result of adding 1 to the height of each vertex in this bottom row).
Topple all vertices in $Q_1$ and then let $P_1$ be the set of all unstable vertices in the top row of $D_{m,n}$.
Topple all vertices in $P_1$ and then let $Q_2$ be the set of all unstable vertices in the bottom row of $D_{m,n}$.
Continue this process until all vertices of $D_{m,n}$ are stable. Let
\begin{align*}
\CanonTop(u)&=(Q_1,P_1,Q_2,P_2,\ldots), \mbox{ and }\\
\canontop(u)&=(|Q_1|,|P_1|,|Q_2|,|P_2|,\ldots).
\end{align*}
We note that it is possible for $\CanonTop(u)=\canontop(u)=()$, the empty sequence. 
Another observation is that if $\CanonTop(u)$ is the empty sequence, then $u$ is not recurrent.
\begin{example}
In this example we will be considering configurations $u=(u_1,\ldots,u_6)$ on the graph $D_{3,4}$.
The assignment of heights to vertices is illustrated in the following diagram. A star at vertex $v_0$ indicates the non-existent height of the sink.
\ \\[0.5em]
\centerline{\includegraphics{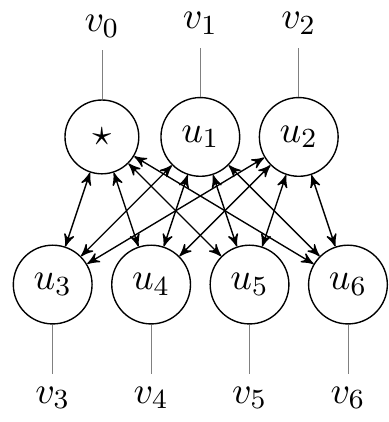}}
\begin{enumerate}
\item[(i)]
Consider $u=(2,0,2,1,0,2) \in \Stable(D_{3,4})$. 
When one grain is added to every vertex in the bottom row, we have the unstable configuration $w=(2,0,3,2,1,3)$. 
Alternate between toppling all unstable vertices in the bottom row and top row:
$$(2,0,3,2,1,3) \to (4,2,0,2,1,0) \to (0,2,1,3,2,1) \to (1,3,1,0,2,1) \in \Stable(D_{3,4}).$$
Since $\sigma(w)=(1,3,1,0,2,1)\neq u$, $u$ is not recurrent.
\ \\[0.5em]
\centerline{\includegraphics{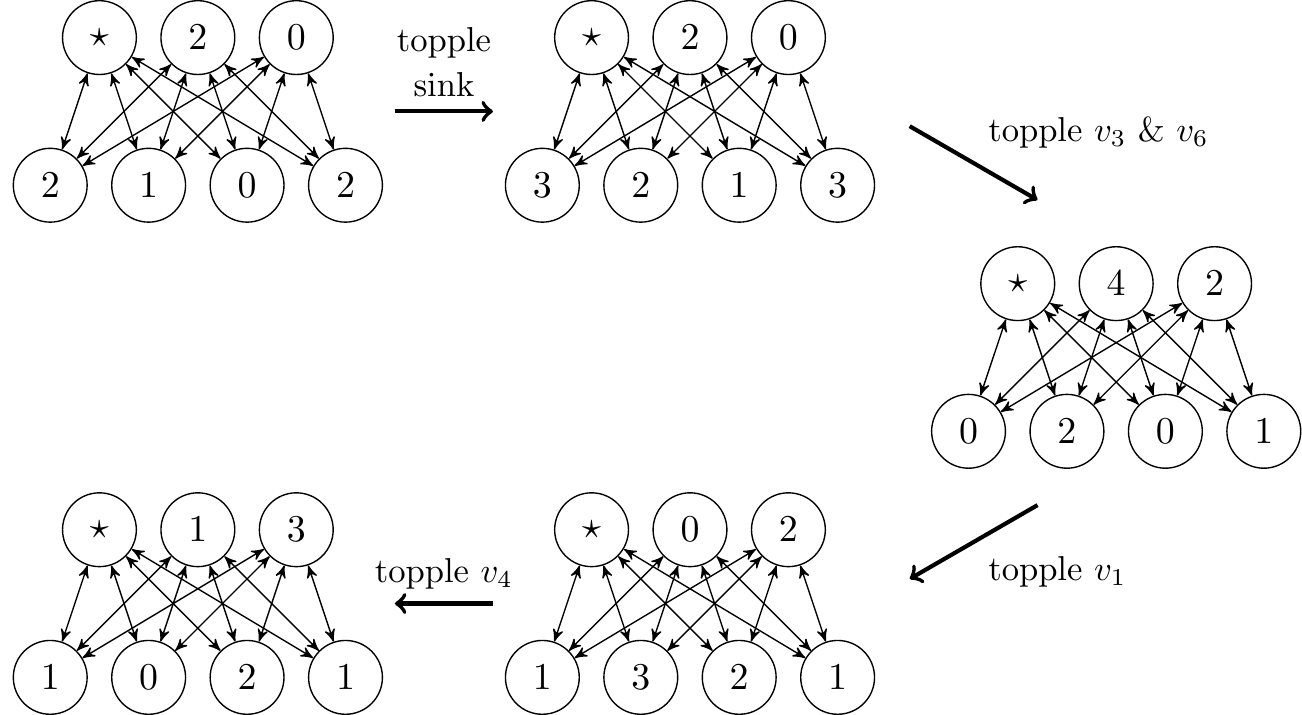}}
\ \\[0.5em]
Vertices $v_3$ and $v_6$ were first to topple in the bottom row, thus $Q_1=\{v_3,v_6\}$.
These caused vertex $v_1$ in the top row to become unstable, so $P_1=\{v_1\}$.
Once vertex $v_1$ toppled, it caused vertex $v_4$ in the bottom row to become unstable, so $Q_2=\{v_4\}$.
When vertex $v_4$ topples it does not cause any vertices in the top row to become unstable and the toppling process now ends.
We have $\CanonTop(u)=(\{3,6\},\{1\},\{4\})$.
\item[(ii)] Consider $u=(0,2,1,2,1,2)\in \Stable(D_{3,4})$. 
Adding one grain to every vertex in the bottom row gives $w=(0,2,2,3,2,3)$, which topples as follows:
$(0,2,2,3,2,3) \to (2,4,2,0,2,0) \to (2,0,3,1,3,1) \to (4,2,0,1,0,1) \to (0,2,1,2,1,2)$.
\ \\[0.5em]
\centerline{\includegraphics{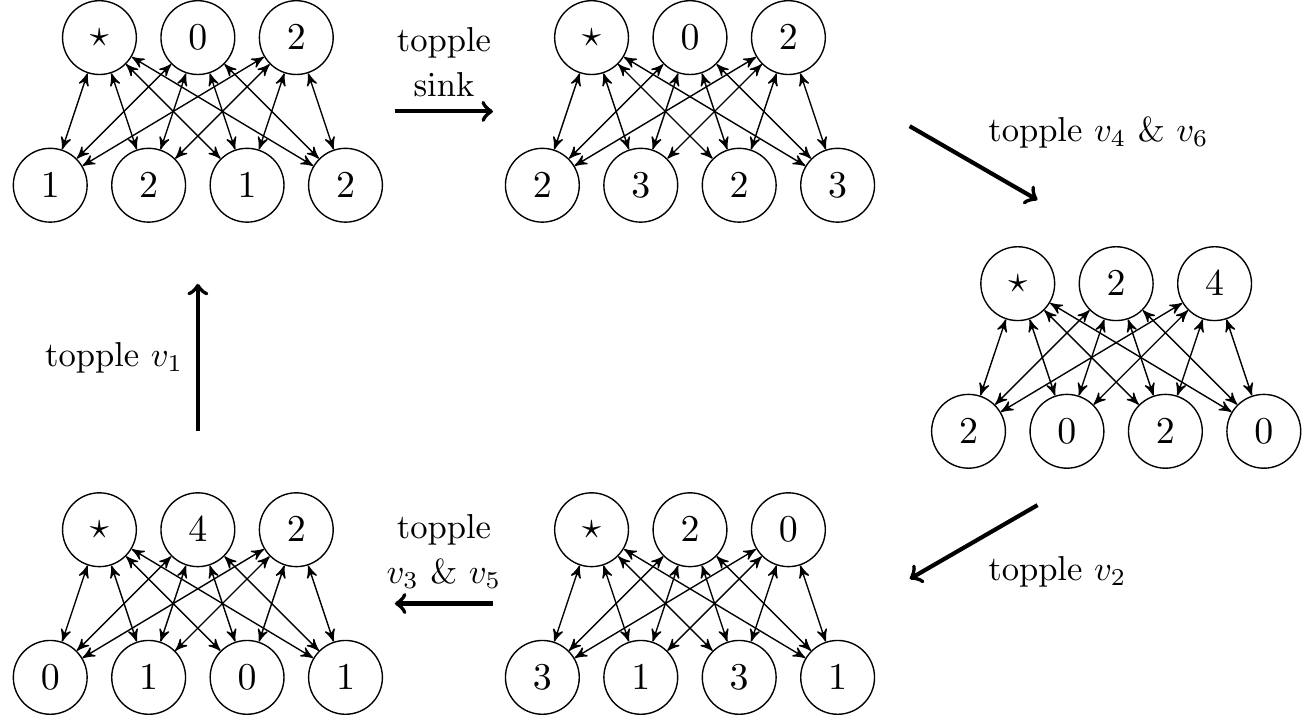}}
\ \\[0.5em]
Since $\sigma(w) =u$ we have that $u \in \Rec(D_{3,4})$.
The toppling process is $\CanonTop(u)=(\{4,6\},\{2\},\{3,5\},\{1\})$.
\end{enumerate}
\end{example}

In the next theorem we will classify the set $\Rec(D_{m,n})$, however some terminology is required first.

\begin{definition}
Given $u=(u_1,\ldots,u_{m+n-1}) \in \Stable(D_{m,n})$ with $\CanonTop(u)=(Q_1,P_1,Q_2,P_2,\ldots)$ and
$\inc_{m,n}(u)=(a_1,\ldots,a_{m-1},b_1,\ldots,b_n)$,
define $f_{m,n}$ to be the collection of cells in the plane given by;
\[ 
f_{m,n}(u) = \left([m-1,m]\times[0,n]\cup \bigcup_{i=1}^{m-1} [i-1,i]\times[0,1+a_i] \right)
\cap \left( \bigcup_{i=1}^{n} [0,1+b_i]\times [i-1,i] \right).
\]
\end{definition}

An alternative definition for $f_{m,n}(u)$ in terms of Young diagrams may be given as follows:

\begin{definition}
Given $u=(u_1,\ldots,u_{m+n-1}) \in \Stable(D_{m,n})$ with $\CanonTop(u)=(Q_1,P_1,Q_2,P_2,\ldots)$ and
$\inc_{m,n}(u)=(a_1,\ldots,a_{m-1},b_1,\ldots,b_n)$,
define $f_{m,n}(u)$ to be the intersection of the Young diagram whose corner at 
$(m,0)$ has heights $(n,1+a_{m-1},\ldots,1+a_1)$ from right to left, and the Young 
diagram whose corner at $(0,n)$ has widths $(1+b_n,\ldots,1+b_1)$ from top to bottom.
\end{definition}

\begin{example}
Given $u=(2,1,2,0,0,2,6,1,5,1) \in \Stable(D_{7,4})$, we have $$\inc_{7,4}(u)=(a_1,\ldots,a_6,b_1,\ldots,b_4)=(0,0,1,2,2,2,1,1,5,6).$$ 
The collections of cells $f_{7,4}(u)$ is\\[0.5em]
        \centerline{\includegraphics{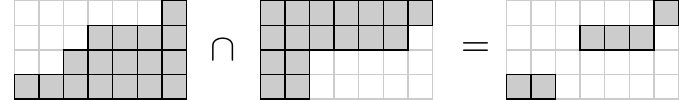}}
Since $f_{7,4}(u) \notin \Para_{7,4}$, the next theorem proves that $u$ is not recurrent.
\end{example}

\begin{theorem} \label{main:thm}
Let $u =(u_1,\ldots,u_{m+n-1}) \in \Stable(D_{m,n})$.
\begin{enumerate}
\item[(a)] $u \in \Rec(D_{m,n})$ iff $f_{m,n}(u) \in \Para_{m,n}$.
\item[(b)] If $u \in \Rec(D_{m,n})$ then $\bounce(f_{m,n}(u))=\canontop(u)$.
\end{enumerate}
\end{theorem}

\begin{proof}
Let $u =(u_1,\ldots,u_{m+n-1}) \in \Stable(D_{m,n})$.
It is a consequence of Lemma~\ref{knuthcori} that in testing whether $u$ is a recurrent configuration, 
every vertex in $\{v_1,\ldots,v_{m+n-1}\}$ topples exactly once during the process. 
The configuration $(u_1,\ldots,u_{m-1},1+u_m,\ldots,1+u_{m+n-1})$ may be interpreted as the result 
of the toppling of the sink $v_0$ in the configuration $u$ even if the height of this vertex is not specified. 
If every other vertex in $\{v_1,\ldots,v_{m+n-1}\}$ topples exactly once, 
then the height of vertex $v_i$ will be $u_i+\left(\sum_{j}e_{j,i}\right) - n = u_i$ for $1\leq i\leq m-1$,
and $1+u_i+\left(\sum_{j\neq 0}e_{j,i}\right) - m =u_i$ for all $m\leq i \leq m+n-1$.
The order of topplings is irrelevant, and in what follows we will adhere to the canonical toppling process mentioned earlier in this section.

The state $u$ is recurrent iff $u'=\inc_{m,n}(u)$ is recurrent, since $\inc_{m,n}(u)$ is simply a reordering of heights within the top and bottom rows of the graph.
If $\CanonTop(u)=(A_1,A_2,\ldots )$ and $\CanonTop(u')=(A'_1 , A'_2,\ldots )$ then $A_i = \{ \pi(j) ~:~ j \in A'_i\}$ for all $i\geq 1$, 
where $\pi=\sperm_{m,n}(u)$.
Consequently, we have $\canontop_{m,n}(u)=\canontop_{m,n}(u')$.

Dhar's criterion (\ref{topeqn}) applied to $u'$ is equivalent to there existing non-empty ordered partitions 
$$P=(P'_1,\ldots,P'_k) \mbox{ and }Q=(Q'_1,\ldots,Q'_{k+1})$$ of the sets 
$\{v_1,\ldots,v_{m-1}\}$ and $\{v_{m},\ldots,v_{m+n-1}\}$, respectively, with the following properties:
\begin{enumerate}
\item[(i')] $1+u'_i = m$ for all $v_i \in Q'_1$.
\item[(ii')] $u'_i+q_0+\cdots + q_j\geq n > u'_i+q_0+\cdots+q_{j-1}$ for all $v_i \in P'_j$, $1\leq j\leq k$.
\item[(iii')] $1+u'_i+p_0+\cdots+p_{j}\geq m > 1+u'_i+p_0+\cdots+p_{j-1}$ for all $v_i \in Q'_{j+1}$, $1\leq j \leq k$.
\end{enumerate}
where $p_i=|P'_i|$, $q_i=|Q'_i|$, and $p_0=q_0=0$.

The understanding above is that $Q'_1$ is the set of vertices on the bottom row that topple first, i.e. $Q'_1=\{v_i: 1+u'_i \geq m, 1\leq i \leq m-1\}$.
Then we topple the vertices in the top row that have become unstable due to grains being added from the previous topplings in 
the bottom row, i.e. $P'_1=\{v_i: 1+u'_i+q_1\geq n, m\leq i \leq m+n-1\}$.
Next we topple those vertices in the bottom row that were stable before the topplings of the vertices $P'_1$, but have consequently become unstable, i.e.
$Q'_2=\{v_i: 1+u_i<n, 1+u'_i+p_1\geq n,  m\leq i \leq m+n-1\}$, and so on.
The block topplings happen via $Q'_1\to P'_1\to Q'_2\to P'_2 \to Q'_3 \cdots \to Q'_{k+1}$. All of these sets except for $Q'_{k+1}$ must be non-empty. 

Since the initial heights of the vertices in the top row of $u'$ are weakly increasing, 
the vertices may be partitioned into contiguous blocks according to the `time' at which they topple. 
The same is true of the bottom row. 
Indeed $p_0,\ldots,p_k$ and $q_0,\ldots,q_k$ are such that
\begin{align*}
v_i \in P'_j \; &\Leftrightarrow \; i \in \{m-(p_0+\cdots+p_j),\ldots,m-1-(p_0+\cdots+p_{j-1})\}\\
v_i \in Q'_j \; &\Leftrightarrow \; i \in \{m+n-(q_0+\cdots+q_j),\ldots,m+n-1-(q_0+\cdots+q_{j-1})\}.
\end{align*}
with
\begin{align*}
p_0+\cdots+p_k &= m-1\\
q_0+\cdots +q_{k+1} &= n\\
p_1,\ldots,p_k,q_1,\ldots,q_{k} &> 0.
\end{align*}
Note that it is possible that $q_{k+1}=0$.
Let us now change variables by setting $a_i=u'_i$ for all $1\leq i\leq m-1$, $b_i=u'_{m-1+i}$ for all $1\leq i \leq n$, 
$x_i=m-1-(p_0+\cdots+p_{k-i})$ and $y_i=n-(q_0+\cdots + q_{k+1-i})$ for all suitably defined $i$.
Conditions (i')--(iii') are equivalent to the following:

There exist numbers $0=x_0<x_1<\cdots<x_{k}=m-1$ and $0=y_0\leq y_1<\cdots <y_{k+1}=n$ such that:
\begin{enumerate}
\item[(i'')] $1+b_i=m$ for all $1+y_k \leq i \leq y_{k+1}=n$.
\item[(ii'')] $1+a_i \in \{1+y_{j+1},\ldots,y_{j+2}\}$ for all $1+x_j\leq i \leq x_{j+1}$ and $0\leq j \leq k-1$.
\item[(iii'')] $1+b_i \in \{1+x_j,\ldots,x_{j+1}\}$ for all $1+y_{j}\leq i \leq y_{j+1}]$ and $0\leq j \leq k-1$.
\end{enumerate}
By defining $x_{k+1}=m$, conditions (i'') and (iii'') above may be more compactly written as
$$\mbox{$1+b_i \in \{1+x_j,\ldots,x_{j+1}\}$ for all $1+y_{j}\leq i \leq y_{j+1}$ and $0\leq j \leq k$.}$$
We have now shown that given $u \in \Stable(D_{m,n})$, the configuration $u \in \Rec(D_{m,n})$ iff
there exist sequences $(x_0,\ldots,x_{k+1})$ and $(y_0,\ldots,y_{k+1})$ with
\begin{enumerate}
\item[(i)] $0=x_0<x_1<\cdots<x_{k}=m-1 <x_{k+1}=m$,
\item[(ii)] $0=y_0\leq y_1<y_2<\cdots<y_{k+1}=n$,
\item[(iii)] $1+a_i \in \{1+y_{j+1},\ldots,y_{j+2}\}$ for all $1+x_j\leq i \leq x_{j+1}$ and $0\leq j \leq k-1$,
\item[(iv)] $1+b_i \in \{1+x_j,\ldots,x_{j+1}\}$ for all $1+y_{j}\leq i \leq y_{j+1}$ and $0\leq j \leq k$.
\end{enumerate}
From Characterization~\ref{bpcharact}, these conditions hold precisely when $f_{m,n}(u) \in \Para_{m,n}$. This completes part (a).

The bounce path of such a parallelogram polyomino is $(y_{k+1}-y_k,x_{k}-x_{k-1},y_k-y_{k-1},x_{k-1}-x_{k-2},\ldots) = (q_1,p_1,q_2,p_2,\ldots)$.
Since the vertices topple in the order $Q_1 \to P_1\to \cdots \to Q_{k+1}$ (note that $Q_{k+1}$ can be empty) we have that $\canontop(u)=(q_1,p_1,q_2,p_2,\ldots)=\bounce(f_{m,n}(u))$, 
which concludes part (b).
\end{proof}

One may give a bijection from $\Rec(D_{m,n})$ to the set of all parallelogram polyominoes whose 
bounce paths have been decorated with an ordered set partition as follows:

Associate to every parallelogram polyomino a pair of ordered set partitions 
whose number of parts and size are related to a path within the polyomino that traverses the polyomino by `bouncing off' edges and corners
in a manner that follows.

If $\bounce(\PP)=(c_1,c_2,\ldots)$, then let us write 
$$\ebounce(\PP)=(c_2,c_4,\ldots) \mbox{ and }\obounce(\PP)=(c_1,c_3,\ldots).$$
Given a set $A$ we will write $X=(X_1,X_2,\ldots)  \vDash A$ to mean that $X$ is an ordered set partition of $A$.
We will write $\type(X)=(|X_1|,|X_2|,\ldots)$, the sequence of sizes of the ordered sets in the partition.
For example, if $X=(\{3,4,2\},\{1,5\},\{6,7\}) \vDash \{1,\ldots,7\}$, then $\type(X)=(3,2,2)$.

Let $\SetPara_{m,n}$ be the set of all triples $(\PP,A,B)$ where $\PP \in \Para_{m,n}$, 
$A=(A_1,A_2,\ldots)$ is an ordered set partition of $\{1,\ldots,m-1\}$ where $A_i$ is 
associated with the $i$th run of west steps in $\Bounce(\PP)$, 
and $B=(B_1,B_2,\ldots)$ is an ordered set partition of $\{m,\ldots,m+n-1\}$ where
$B_i$ is associated with the $i$th run of south steps in $\Bounce(\PP)$. 
More formally,
$$\SetPara_{m,n} = \bigcup_{\PP \in \Para_{m,n}} 
		\bigcup_{A \vDash \{1,\ldots,m-1\} \atop \type(A)=\ebounce(\PP)} 
		\bigcup_{B\vDash \{m,\ldots,m+n-1\} \atop \type(B)=\obounce(\PP)} 
		(\PP,A,B).$$
Note that if $(\PP,A=(A_1,A_2,\ldots),B=(B_1,B_2,\ldots)) \in \SetPara_{m,n}$ and $\bounce(\PP)=(\ell_1 ,\ell_2,\ldots)$, then
$\ell_{2i}=|A_i|$ and $\ell_{2i-1}=|B_i|$ for all $i\geq 1$.

We stress that the sizes of the sets $A$ and $B$ in the triple $(\PP,A,B)$ are determined by the bounce path of $\PP$, 
whereas the numbers in the sets which constitute $A$ and $B$ are given and do not come from the polyomino $\PP$.

\begin{example}
Let $\PP$ be $\PP_2$ in Example~\ref{bpexamp}. Let $A=(\{1,4,5\},\{2,3,6\})$ and $B=(\{8\},\{7,9\})$. Then $(\PP,A,B) \in \SetPara_{7,3}$
and may be represented in the following way:
	\begin{center}
	\includegraphics{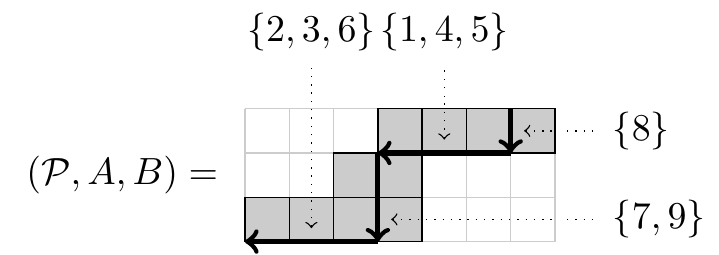}
	\end{center}
\end{example}

Given $u \in \Stable(D_{m,n})$ with $\CanonTop(u)=(Q_1,P_1,Q_2,P_2,\ldots)$, let $$\RtoSP(u)=(f_{m,n}(u),(P_1,P_2,\ldots),(Q_1,Q_2,\ldots)).$$
This construction, combined with parts (a) and (b) of Theorem~\ref{main:thm}, gives:

\begin{corollary}
$\RtoSP : \Rec(D_{m,n}) \to \SetPara_{m,n}$ is a bijection.
\end{corollary}

\newcommand{\Nara}[2]{\mathrm{Nara}(#1,#2)}
It is well-known from the literature on the sandpile model that recurrent configurations on a 
graph $G$ are in one-to-one correspondence with spanning trees of $G$. 
Thus the number of recurrent configurations on $D_{m,n}$ is the number of spanning trees of $K_{m,n}$, which is $m^{n-1} n^{m-1}$.

\begin{corollary}\label{naracor} For all $m,n\geq 1$, $|\Rec(D_{m,n})|=m^{n-1} n^{m-1}$  and $|\RecInc(D_{m,n})|= \Nara{m+n-1}{m}$,
where $\Nara{a}{b}=\frac{1}{a}{a\choose b}{a\choose b-1}$ are the Narayana numbers~\citep{Narayana:1955}.
\end{corollary}

\section{Special types of minimal recurrent configurations}
\label{SEC:4}
In this section we examine several different types of minimal recurrent configurations 
in $\Rec(D_{m,n})$ under the map $f$ that was given in the previous section.
As was mentioned before, the relationship between these different collections is summarised in Figure~\ref{SPECIAL:TYPES}.
\begin{figure}[h!] \centerline{\includegraphics{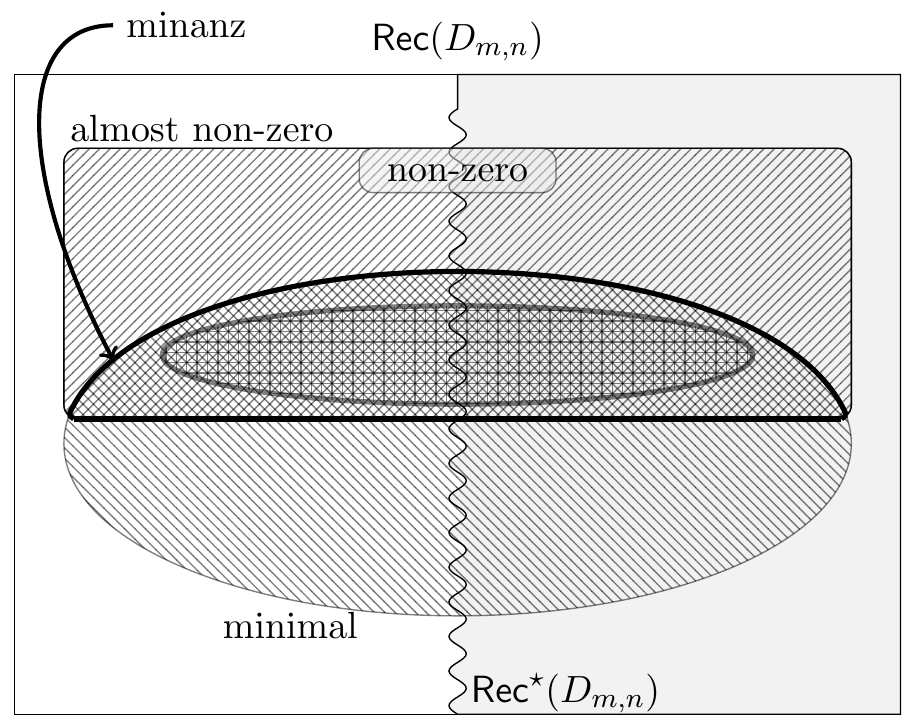}} \caption{\label{SPECIAL:TYPES} How the different types of recurrent configurations of Section \ref{SEC:4} relate to one another. The intersection of the line-filled ellipse with the line-filled rectangular region is the collection of ``minanz'' recurrent configurations. When $m=n$ this is the set $\SqRec_n$ defined and discussed in Section \ref{SEC:4}. When $m=n$, the elliptical region contained within the minanz region is the collection of square top-heavy minanz recurrent configurations $\THSqRec_n$, also defined and discussed in Section \ref{SEC:4}.} \end{figure}
Recurrent configurations which contain no empty vertices are a natural class to examine from a physical viewpoint. 
However, if a recurrent configuration is minimal, then there must be at least one vertex that is empty (this will be shown in Lemma~\ref{minanzlemma}).
So in what follows we will be looking at minimal recurrent configurations that are almost non-empty, which we will call {\it{minanz}} configurations, 
in the sense that {\it{exactly}} one designated vertex on the bottom row is allowed to be empty.

In Section \ref{SEC:4}.2, we look at configurations that have the same number of vertices on the top row as on the bottom row.
We call such configurations {\it{square configurations}}, i.e. those for which $m=n$.
Square configurations in $\Rec(D_{n,n})$ that are {\it{minanz}} are shown to correspond via a mapping $\phi$ to a matrix analogue of set partitions.
After this we consider a subset of the aforementioned square configurations, which we will call {\it{top-heavy}}. The image of these configurations under $\phi$ are precisely those matrices which are upper-triangular (Theorem~\ref{minanzmatrix}).
This leads to a connection with (2+2)-free partially ordered sets, which is explained towards the end of this section (Theorem~\ref{recs_to_posets}).

Given $u\in \Rec(D_{m,n})$, let 
$$\level(u) = u_1+\cdots+u_{m+n-1} - n(m-1).$$
Define the equivalence relation $\sim$ as follows: for $u, u' \in \Rec(D_{m,n})$, we write $u\sim u'$ if $\inc_{m,n}(u)=\inc_{m,n}(u')$.
The quotient set $\Rec(D_{m,n})/ \sim$ is in one-to-one correspondence with $\RecInc(D_{m,n})$, the set of all increasing configurations of $\Rec(D_{m,n})$.

\begin{theorem}\label{levelareaequiv}
If $u \in \Rec(D_{m,n})$, then $\level(u)=\area(f_{m,n}(u))-(m+n-1)$.
\end{theorem}

\begin{proof}
Let $u \in \Rec(D_{m,n})$. Then we have that 
\[ n+\sum_{i=1}^{m-1} (1+u_i)=n+\sum_{i=1}^{m-1} (1+a_i) = \area(f_{m,n}(u))+\larea(f_{m,n}(u)) \]
and
\[ \sum_{j=m}^{m+n-1}(1+u_j)=\sum_{j=1}^n (1+b_j)=\area(f_{m,n}(u))+\uarea(f_{m,n}(u)).\]
Adding both of these gives 
\begin{align*}
n+\sum_{i=1}^{m+n-1} (1+u_i) 
&= \uarea(f_{m,n}(u))+2\area(f_{m,n}(u))+\larea(f_{m,n}(u))\\
&= mn+\area(f_{m,n}(u)),
\end{align*}
which gives $\level(u)=\area(f_{m,n}(u))-(m+n-1)$.
\end{proof}

\subsection{Minimal and minanz recurrent configurations}
A configuration $u \in \Rec(D_{m,n})$ for which $\level(u)$ is as small as possible, i.e. 0, is called {\it{minimal}}. 
From Theorem~\ref{levelareaequiv}, $\level(u)=0$ iff $\area(f_{m,n}(u))=m+n-1$, and this happens iff $f_{m,n}(u) \in \Ribbon_{m,n}$.

\begin{theorem}
Let $u \in \Rec(D_{m,n})$.
The following are equivalent: 
\begin{enumerate}
\item[(i)] $u$ is minimal, 
\item[(ii)] $f_{m,n}(u) \in \Ribbon_{m,n}$, and 
\item[(iii)] $\level(u)=0$.
\end{enumerate}
\end{theorem}

The set of all different minimal configurations $\{u \in \RecInc(D_{m,n})~:~ u \mbox{ is minimal}\}$ is in one-to-one correspondence with the set of all ribbon polyominoes $\Ribbon_{m,n}$.
Every ribbon polyomino is uniquely characterised by a lattice path that goes through all the centres of its cells. Such a path goes from $(1/2,1/2)$ to $(m-1/2,n-1/2)$ and takes unit north and east steps.

\begin{corollary}
$|\{u \in \RecInc(D_{m,n})~:~ u \mbox{ is minimal}\}|={m+n-2 \choose m-1}$.
\end{corollary}

We call a configuration $u \in \Rec(D_{m,n})$ {\it{almost non-zero}} if $$u_1,\ldots,u_{m-1},u_{m+1},\ldots,u_{m+n-1}>0.$$
Configurations that are both minimal and almost non-zero will be called {\it{minanz configurations}}.

\begin{lemma}\label{minanzlemma}
If $u \in \Rec(D_{m,n})$ is minanz, then 
\begin{enumerate}
\item[(i)] $u_m=0$,
\item[(ii)] the cells $[0,1]\times [0,1]$, $[0,1]\times [1,2]$ and $[1,2]\times [1,2]$ are in $f_{m,n}(u)$, and
\item[(iii)] $\Bounce(f_{m,n}(u))$ has suffix $\ww\sss$.
\end{enumerate}
\end{lemma}

\begin{proof}
If $u \in \Rec(D_{m,n})$ is minimal, then $\PP=f_{m,n}(u)$ is a ribbon polyomino.
A ribbon polyomino contains either the cell $[1,2]\times[0,1]$ or the cell $[0,1]\times [1,2]$, but not both.
Since $u_1,\ldots,u_{m-1}>0$, we have that $1+u_1,\ldots,1+u_{m-1}\geq 2$.
Thus $[0,1]\times [1,2]$ is in $\PP$ and $[1,2]\times[0,1]$ is not, which implies the smallest entry of $u_m,\ldots,u_{m+n-1}$ is 0.
Since $u_{m+1},\ldots,u_{m+n-1}>0$, we must have $u_m=0$ and also that $[1,2]\times[1,2] \in \PP$.

For a general $u\in \Rec(D_{m,n})$, $\Bounce(f_{m,n}(u))$ may end with either of the four suffixes: $\sss\sss$, $\sss\ww$, $\ww\sss$, or $\ww\ww$. 
Suppose that $\Bounce(f_{m,n}(u))$ ends in $\ww\ww$. One must then have $u_m>0$, which contradicts (i) above.
If $\Bounce(f_{m,n}(u))$ ends with $\sss\ww$, then consider the third last step. The suffix of $\Bounce(f_{m,n}(u))$ is either (a) $\sss\sss\ww$ or (b) $\ww\sss\ww$. 
Neither (a) nor (b) can be bounce paths in the final cells of $\Bounce(f_{m,n}(u))$ given in (ii).
Finally, suppose that $\Bounce(f_{m,n}(u))$ ends in $\sss\sss$. The two steps preceding $\sss\sss$ must then be $\ww\ww$. 
Considering the position of the leftmost cell at height $[2,3]$ shows that no such path exists with suffix $\ww\ww\sss\sss$.
\end{proof}

For a general graph $G$, minimal recurrent configurations on $G$ are in bijection with acyclic orientations of $G$ 
whose unique sink is a distinguished vertex, as is the sink in the sandpile model we are considering. 
Our restriction to minanz configurations corresponds to acyclic orientations where in addition there is a unique fixed source that shares an edge with the sink. (See Gioan and Las~Vergnas \citep{GioanLasVergnas:2005}.)

It is a consequence of (iii) and the identification of the bounce path with $\canontop(u)$ that if $u \in \Rec(D_{m,n})$ is minanz with $\CanonTop(u)=(Q_1,P_1,\ldots ,Q_{k+1})$, then $Q_{k+1}=\{m\}$.

\begin{corollary}
$|\{u \in \RecInc(D_{m,n})~:~ u \mbox{ is }\minanz\}|={m+n-4 \choose m-2}$.
\end{corollary}

Next we consider the special case for when $m=n$. 

\subsection{Square minanz and square top-heavy minanz recurrent configurations}
Let $\BiComp_n$ be the set of square matrices whose entries partition the set 
$\{1,\ldots,n\}$ and having no rows or columns consisting of only empty sets.  
Let $\UTBiComp_n$ be the collection of upper triangular matrices in $\BiComp_n$.  Define
\begin{align*}
\SqRec_n &= \{u \in \Rec(D_{n,n})~:~ u\mbox{ is }\minanz\}.
\end{align*}

Suppose that $u \in \SqRec_n$ with $\CanonTop(u) = (Q_1,P_1,\ldots,Q_{k+1}=\{n\})$.
If vertex $v_x$ is such that $x \in P_i \cup Q_i$ then we say that {\it{$v_x$ is in the $i$th wave of $n$}} and denote this by
$\wave_u(v_x) = i$.
Let us define
\[ \THSqRec_n= \{u \in \SqRec_n~:~ \wave_u(v_{n+x}) \leq \wave_u(v_x)\mbox{ for all $1\leq x<n$}\}.\]
In words, $\THSqRec_n$ is the collection of all square minanz configurations whose canonical topplings have the property 
that bottom vertices topple on or before the wave which topples the vertex directly above.

From the definition of $\THSqRec_n$, and using the fact that $f(\THSqRec_n) \subseteq \Ribbon_{n,n}$, 
the image $f(\THSqRec_n)$ is precisely the class of ribbon polyominoes whose cells/squares are such 
that none of their centres lie beneath the line $y=x$ in the plane.

Given $u \in \Rec(D_{n,n})$ is minanz with $\CanonTop(u)=(Q_1,P_1,\ldots ,Q_{k+1}=\{n\})$, define $Q_i'=Q_i-n$ for all $i$.
So $(Q_1',\ldots,Q_k')$ and $(P_1,\ldots,P_k)$ are both ordered set partitions of $\{1,\ldots,n-1\}$.
Let $\ttm(u)$ be the $k\times k$ matrix $M$ with $M_{ij}=P_i \cap Q_j'$.
Equivalently, we have that $x \in M_{ij}$ if $x \in P_i$ and $n+x \in Q_j$ for all $x \in \{1,\ldots,n-1\}$.

\begin{example}
\begin{enumerate}
\item[(i)] Consider $u=(4,5,6,1,4,5,4,0,7,1,1,4,6,1,7) \in \SqRec_{8}$.  We have 
$$\CanonTop(u)=(\{9,15\}, \{3\},\{13\},\{2,6\},\{12\},\{1,5,7\},\{10,11,14\},\{4\},\{8\})$$
so that
\begin{align*}
(P_1,P_2,P_3,P_4) &= (\{3\},\{2,6\},\{1,5,7\},\{4\}) \\
(Q_1,Q_2,Q_3,Q_4) &= (\{9,15\},\{13\},\{12\},\{10,11,14\})\\
\Rightarrow (Q_1',Q_2',Q_3',Q_4') &= (\{1,7\},\{5\},\{4\},\{2,3,6\}).
\end{align*}
Thus $$\ttm(u) =
\left(
\begin{matrix}
\emptyset & \emptyset & \emptyset & \{3\} \\
\emptyset & \emptyset & \emptyset & \{2,6\} \\
\{1,7\} & \{5\} & \emptyset & \emptyset \\
\emptyset & \emptyset & \{4\} & \emptyset 
\end{matrix}
\right).  $$
\item[(ii)]
Consider $u=(4,3,4,1,0,2,1,4,1) \in \SqRec_{5}$. We have $$\CanonTop(u)=(\{8\},\{1,3\},\{6\},\{2\},\{7,9\},\{4\},\{5\}).$$
From this $(P_1,P_2,P_3)=(\{1,3\},\{2\},\{4\})$, $(Q_1,Q_2,Q_3)=(\{8\},\{6\},\{7,9\})$ and so $(Q_1',Q_2',Q_3')=(\{3\},\{1\},\{2,4\})$.
Thus 
$$\ttm(u) =
\left(
\begin{matrix}
\{3\} & \{1\} & \emptyset \\
\emptyset & \emptyset & \{2\} \\
\emptyset & \emptyset & \{4\} 
\end{matrix}
\right).  $$
\end{enumerate}
\end{example}

\begin{theorem} \label{minanzmatrix}
{\protect{\rm{(i)}}} $\ttm:\SqRec_n \to \BiComp_n$ is a bijection.
{\protect{\rm{(ii)}}} $\ttm:\THSqRec_n \to \UTBiComp_n$ is a bijection.
\end{theorem}

\begin{proof}
For (i) let $u \in \SqRec_n$. We have $\CanonTop(u) = (Q_1,P_1,\ldots,Q_{k+1}=\{n\})$.
From the nature of the toppling process
$(Q_1',\ldots,Q_k') \vDash \{1,\ldots ,n-1\} \mbox{ and } (P_1,\ldots,P_k) \vDash \{1,\ldots,n-1\}$.
Since every element $x \in \{1,\ldots,n-1\}$ appears in exactly one of the $P$ sets, $P_i$ say, and appears in exactly one of the $Q'$ sets, $Q_j'$ say,
we have that $x \in M_{ij}=P_i \cap Q_j'$. There can be no row or column of empty sets for this reason. Thus $M = \ttm(u) \in \BiComp_n$.

Let us now define a function $\mu$ on $\BiComp_n$.
Suppose that $M \in \BiComp_n$ with $\dim(M)=k$. 
For all $i \in \{1,\ldots,k\}$, let 
\begin{align*}
P_i &= M_{i1} \cup \ldots \cup M_{ik} \\
Q_i'&= M_{1i} \cup \ldots \cup M_{ki}.
\end{align*}
Define $Q_i= Q_i' + n$ for all $i \in \{1,\ldots,k\}$ and $Q_{k+1}=\{n\}$.
Let $p_0=q_0=0$, $p_i=|P_i|$ and $q_i=|Q_i|$.
Let $u = \mu(M)$ be the configuration where $u_n=0$,
$$
\begin{array}{lcll}
u_{n+x} &=& 
n-1-(p_0+\cdots + p_{i-1}) & \mbox{ for all } x \in Q_i\mbox{ and } i \in \{1,\ldots,k\} \\
u_x &=& n-1-(q_0+\cdots + q_{i}) & \mbox{ for all } x \in P_i \mbox{ and } i \in \{1,\ldots,k\}.
\end{array}
$$
The value $\level(u)$ is minimal. 
Since $Q_{k+1}=\{n\}$, $u_n=n-1-(p_0+\cdots+p_k)=0$, and $u_i>0$ for all other $i$, hence $u$ is almost non-zero. 
Thus $u=\mu(M) \in \SqRec_n$.
It is straightforward to check that $\ttm(\mu(M))=M$ and $\mu(\ttm(u))=u$ for all $M \in \BiComp_n$ and $u \in \SqRec_n$, respectively. 
Thus $\mu=\ttm^{-1}$ and $\ttm:\SqRec_n \to \BiComp_n$ is a bijection.

Part (ii): from part (i) we have that $\ttm(\THSqRec_n) \subseteq \BiComp_n$.
Let $u \in \SqRec_n$, $M=\ttm(u)$ and suppose that $\dim(M)=k$. The matrix $M$ is upper-triangular (i.e. $M \in \BiComp_n$) iff
$Q_i' \subseteq P_1\cup \cdots \cup P_i$ for all $1\leq i \leq k$.
This condition is equivalent to `$\wave_u(v_{n+x})=i$ $\Rightarrow$ $\wave_u(v_x)\leq i$ for all $1\leq i \leq k$ and $1\leq x \leq n-1$', which in turn is equivalent to `$\wave_u(v_{n+x}) \leq \wave_u(v_x)$ for all $1\leq x \leq n-1$.
\end{proof}

\begin{corollary}
$|\{u \in \Rec(D_{n,n})~:~ u\mbox{ is }\minanz\}|~ =~ \sum_k S(n-1,k)^2$, where $S(n,k)$ are the Stirling numbers of the 2nd kind.
\end{corollary}

As we mentioned in the introduction, the matrices $\UTBiComp_n$ were the subject of the two papers 
~\citep{ClaessonDukesKubitzke:2011,DukesJelinekKubitzke:2011} which related 
these matrices to (2+2)-free partially ordered sets (posets) on the set $\{1,\ldots,n\}$ (also known as interval orders).
There are several other rich connections between these objects and labelled Stoimenow matchings, ascent sequences and pattern avoiding permutations.
The remainder of this section is dedicated to showing a direct link between (2+2)-free posets and the recurrent configurations of $\THSqRec_n$.
This link was the original motivation behind this paper.

Let $P=(P,\preceq_P)$ be a partially ordered set on $\{1,\ldots,n\}$.  We denote by $P^{\star}$ the dual poset of $P$.
Given $x \in P$, let $D(x)=\{y \in P: y \prec x\}$ be the down set of $x$. 
A defining property of (2+2)-free posets is the following: a poset $P$ is (2+2)-free iff the set of all down-sets of $P$ may be linearly ordered by inclusion.
Let $\Posets_n$ be the set of all (2+2)-free posets on the set $\{1,\ldots,n\}$.
Every $P \in \Posets_n$ can be uniquely represented as a sequence of sets
$P=(D_0(P),\ldots,D_{k-1}(P); L_0(P),\ldots,L_{k-1}(P))$ where $\emptyset = D_0(P) \subset D_1(P)\subset \cdots \subset D_{k-1}(P)$
and $(L_0(P),\ldots,L_{k-1}(P))$ is an ordered set partition of $\{1,\ldots,n\}$ such that if $x \in L_i(P)$ then $D(x)=D_i(P)$.

\begin{definition}
Given $M \in \UTBiComp_n$, let $P=(P,\preceq_P)=\mtp(M)$ be the partially ordered set 
on $\{1,\ldots,n\}$ where $x\prec_P y$ iff there exist $i,j,i',j'$ with
$x \in M_{ij}$, $y \in M_{i'j'}$ and $j<i'$.
\end{definition}

We refer the reader to the papers ~\citep{ClaessonDukesKubitzke:2011,DukesJelinekKubitzke:2011} if there is any confusion surrounding the definitions or terminology.

\begin{example}\label{mtp_example}
Let $M=\left(\begin{matrix} 
\{3\} & \{7,2\} & \emptyset \\
\emptyset & \{5\} & \{1\} \\
\emptyset & \emptyset & \{4,6\}
\end{matrix}\right)$. Then $P=(P,\preceq)$ is the poset on $\{1,2,\ldots,7\}$ with the following relations:
$3\prec_P 1,4,5,6$ and $2,5,7 \prec_P 4,6$.

\centerline{\includegraphics{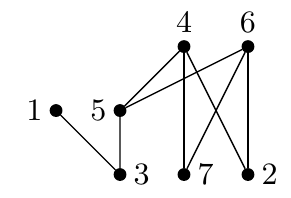}}
\end{example}

\begin{theorem}[Dukes et al.~\citep{DukesJelinekKubitzke:2011}]
Let $\Posets_n$ be the set of all (2+2)-free posets on $\{1,\ldots,n\}$. Then
$\mtp:\UTBiComp_n \to \Posets_n$ is a bijection.
\end{theorem}

In the next theorem we see that the composition of the functions $\ttm$ and $\mtp$ is such that the recurrent configuration can be easily read from the level structure of the associated poset.

\begin{theorem} \label{recs_to_posets}
Let $u \in \THSqRec_n$ with $\CanonTop(u)=(Q_1,P_1,\ldots,Q_{k+1}=\{n\})$ and $(P,\preceq)=\mtp(\ttm(u))$.  Then 
$$
\begin{array}{lcll}
u_{x+n} &=& 
|D_{j+1}(P^{\star})| & \mbox{ for all } x \in L_j(P^{\star}) \mbox{ and }j \in \{0,\ldots,k-1\} \\
u_{x} &=& 
n-|D_{j+1}(P)| & \mbox{ for all } x \in L_j(P) \mbox{ and }j \in \{0,\ldots , k-1\} .
\end{array}
$$
\end{theorem}

\begin{proof}
Let $u \in \THSqRec_n$ with $\CanonTop(u)=(Q_1,P_1,\ldots,Q_{k+1}=\{n\})$. Let $M=\ttm(u)$ and $(P,\preceq)=\mtp(M)$.
We have $\dim(M)=k$, $P_i$ is the union of entries in row $i$ of $M$, and $Q_i$ is the union of entries in column $i$ of $M$ with all entries increased by $n$, for all $i \in \{1,\ldots,k\}$.
The poset $P=(P,\preceq_P)$ is specified by;
$$P=(D_0(P),\ldots,D_{k-1}(P);L_0(P),\ldots,L_{k-1}(P)).$$
The dual $P^{\star}$ of $P$ is given by 
$$P^{\star}=(D_0(P^{\star}),\ldots,D_{k-1}(P^{\star});L_0(P^{\star}),\ldots,L_{k-1}(P^{\star}))$$
where $L_{k-i-1}(P^{\star}) = D_{i+1}(P)\backslash D_i(P)$ and $D_{k-1-i}(P^{\star}) = L_{i+1}(P)\cup \ldots \cup L_{k-1}(P)$
for all $i \in \{0,\ldots ,k-1\}$. We use the convention $D_k(P)=\{1,\ldots,n-1\}$.

From Theorem~\ref{minanzmatrix} we have $u_{n}=0$, 
$$
\begin{array}{lcll}
u_{x+n} &=& 
n-1-(p_0+\cdots + p_{i-1}) & \mbox{ for all } x \in Q_i\mbox{ and } i \in \{1,\ldots,k\}, \mbox{ and} \\
u_x &=& n-1-(q_0+\cdots + q_{i}) & \mbox{ for all } x \in P_i \mbox{ and } i \in \{1,\ldots,k\}.
\end{array}
$$
For all $i \in \{1,\ldots,k\}$, $P_i=L_{i-1}(P)$ and $Q_i= D_i(P)\backslash D_{i-1}(P)=L_{k-i}(P^{\star})$.
For all $i \in \{1,\ldots,k\}$, 
\begin{align*}
n-1-(p_0+\cdots+p_{i-1}) 
&= p_i+\cdots +p_{k} \\
&= |L_{i-1}(P)\cup \ldots \cup L_{k-1}(P)| \\
&= |D_{k-1-i}(P^{\star})|\\
\noalign{and}
n-1-(q_0+\cdots+q_i) 
&= q_{i+1}+\cdots + q_k + q_{k+1} \\
&= |Q_{i+1}\cup \ldots \cup Q_{k}| + q_{k+1} \\
&= |(D_{i+1}(P)\backslash D_i(P)) \cup \ldots \cup (D_k(P)\backslash D_{k-1}(P))| +1 \\
&= |D_k(P)\backslash D_{i}(P)|+1 \\
&= n-1-|D_i(P)|+1 \;=\; n-|D_i(P)|.\qedhere
\end{align*}
\end{proof}

\begin{example}
Consider the poset $P$ given in Example~\ref{mtp_example}. The level- and down-sets of this poset are illustrated in the following diagram:\\[1em]
\centerline{\includegraphics{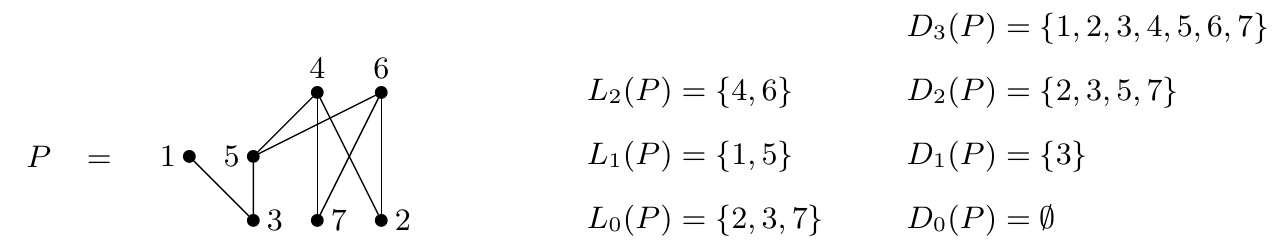}}
\ \\
From Theorem~\ref{recs_to_posets} we have that $u_x=|D_{j+1}(P)|$ for all $x \in L_j(P)$ and $j \in \{0,1,2\}$, hence
$u_2=u_3=u_7 = 8-1=7$, $u_1=u_5=8-4=4$, $u_3=7$ $\Rightarrow$ $(u_1,\ldots,u_7)= (4,7,7,1,4,1,7)$.

The dual $P^{\star}$ of $P$, together with its level- and down-sets:\\[1em]
\centerline{\includegraphics{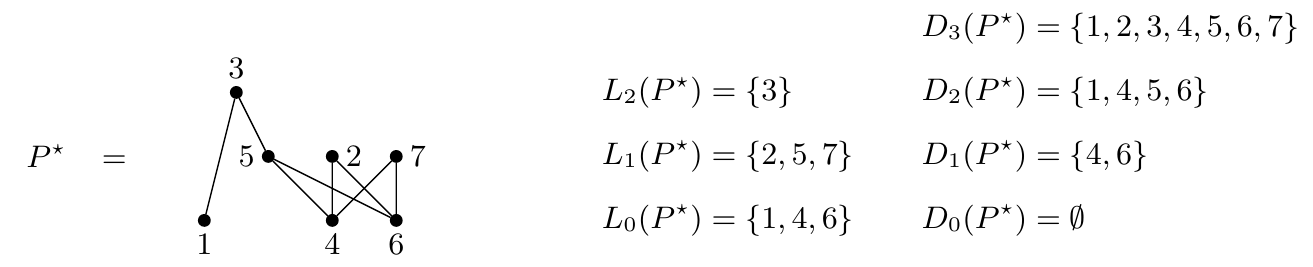}}
\ \\
From Theorem~\ref{recs_to_posets} we have that $u_x=|D_{j+1}(P^{\star})|$ for all $x \in L_j(P^{\star})$ and $j \in \{0,1,2\}$, hence
$u_{8+1} =u_{8+4}=u_{8+6}=2$, $u_{8+2}=u_{8+5}=u_{8+7}=4$ and $u_{8+3}=7$.
So $(u_9,\ldots,u_{15})=(2,4,7,2,4,2,4)$.
The configuration is $u=(4,7,7,1,4,1,7,0,2,4,7,2,4,2,4)$.
\end{example}

We end this section with the following conjecture concerning square non-zero recurrent configurations.
Is there a combinatorial explanation for the relationship to walks in the plane? (See \cite[Sequence A145600]{Sloane:2011}.)

\begin{conjecture} \label{NZConfConj}
Let $a_n=\{u \in \RecInc(D_{n,n})~:~ u_1,\ldots,u_{m+n-1}>0\}|$.
Then $a_n=\frac{1}{n-1} {2n-2\choose n} {2n\choose n-2}$, which is the number of walks from $(0,0)$ to $(0,1)$ 
that remain in the upper half-plane $(y\geq 0)$ using $2n-3$ unit steps $\{\nn,\sss,\ee,\ww\}$  
\end{conjecture}

\section{$q,t$-Narayana polynomials and their symmetry}
\label{SEC:5}
In this section we will introduce a polynomial that we call the $q,t$-Narayana polynomial. 
The polynomial is the generating function for the bistatistic $(\area,\bounceWeight)$ on the set of parallelogram polyominoes.
In terms of recurrent configurations of the sandpile model, the area statistic is proportional to the sum of the heights of the piles, 
and the $\bounceWeight$ statistic is related to the canonical toppling process.
When viewed in the parallelogram polyomino, the bounce path is almost identical to the bounce path that Haglund defined for Dyck paths~\citep{Haglund:2003}.
Our polynomial is a natural extension of the area and bounce path statistics to the class of parallelogram polyominoes. 
More will be discussed about this in the subsections that follow.

For any polyomino $\PP \in \Para_{m,n}$, define its \emph{bounce weight} to be 
$$\bounceWeight(\PP) = \sum_{i=1}^{k} \lceil i/2 \rceil c_i,$$ where $\bounce(\PP)=(c_1,\ldots c_k)$.

This weight may also be described by summing weights on each step of the bounce path.
The initial step has a weight of $1$, and the weight of subsequent steps is incremented by $1$ after each turn on the upper boundary path of the polyomino.
Consider the following two polyominoes:

\begin{center}
\begin{tabular}{c} 
\includegraphics{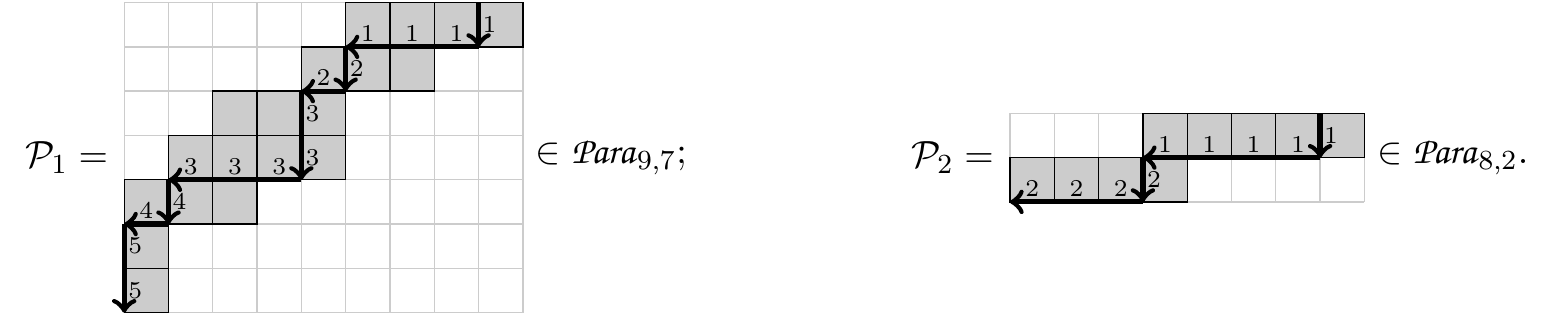}
\end{tabular}
\end{center}

The bounce weights are 
\begin{align*}
\bounceWeight(\PP_1) &= (1)+(1+1+1)+(2)+(2)+(3+3)+(3+3+3)+(4)+(4)+(5+5) = 41,\\
\noalign{and}
\bounceWeight(\PP_2) &= (1)+(1+1+1+1)+(2)+(2+2+2)=13. 
\end{align*}

The distribution of the bistatistic $(\area,\bounceWeight)$ on polyominoes in $\Para_{m,n}$ is represented by the generating function
$$\sgdistribution_{m,n}(q,t) = \sum_{\PP \in \Para_{m,n}}q^{\area(\PP)}t^{\bounceWeight(\PP)}.$$
We call these polynomials $q,t$-Narayana polynomials because $\{F_{m,n}(1,1)\}_{m,n\geq 1}$ 
are the Narayana numbers that were mentioned in Corollary~\ref{naracor}.
The following conjecture has been verified for all pairs $(m,n)$ with $m\leq 11$ and $n \in \mathbb{N}$.

\begin{conjecture} \label{first:sym}
For all positive integers $m$ and $n$, the distribution of the bistatistic $(\area,\bounceWeight)$ on polyominoes in $\Para_{m,n}$ is symmetric:
$$ \sgdistribution_{m,n}(q,t) = \sgdistribution_{m,n}(t,q).$$
\end{conjecture}

Further to this we posit another symmetry (which has been checked for all pairs $(m,n)$ with $\max(m,n)\leq 9$).
\begin{conjecture} \label{second:sym}
For all positive integers $m$ and $n$ we have $ \sgdistribution_{m,n}(q,t) = \sgdistribution_{n,m}(q,t).$
\end{conjecture}

Symmetry along the main diagonal of parallelogram polyominoes provides a bijection from polyominoes of $\Para_{m,n}$ to $\Para_{n,m}$ that preserves the area statistic. This mapping gives the following special case of Conjecture~\ref{second:sym}:

\begin{theorem}
For all positive integers $m$ and $n$ we have $ \sgdistribution_{m,n}(q,1) = \sgdistribution_{n,m}(q,1).$
\end{theorem}

In the next two subsections we will present some arguments that support Conjecture~\ref{first:sym}.
In the final subsection we will discuss the relationship of Conjecture~\ref{first:sym} to the 
sandpile model and a similarly symmetric bistatistic on Dyck paths that was introduced by Haglund. 

\subsection{Proofs of the conjectures for small values of $\min(m,n)$}
In this subsection we show that Conjectures~\ref{first:sym} and \ref{second:sym} are true for all pairs $(m,n)$ with $m=2$.
The same method can be used to show the conjectures are true for other similarly small values of one of the parameters by using regular expressions.
These computations will not be detailed in this paper but we list the first few polynomials in the appendix, from which symmetry is apparent. 
Conjecture~\ref{first:sym} has been checked up to $\min(m,n) = 11$ and $n \in \mathbb{N}$ using generating functions resulting from the classical transfer-matrix method.

We consider polyominoes in $\Para_{2,n}$ for any $n\geq 1$.
Define the generating function for the bistatistic of all these polyominoes, adding a variable $z$ to record the height $n$ of polyominoes:
$$\sgdistribution_{2,\star}(q,t;z) = \sum_{n\geq 1} \sgdistribution_{2,n}(q,t)z^n.$$

\begin{theorem}
$\sgdistribution_{2,\star}(q,t;z) = {t^2q^2z}/{(1-qtz)(1-q^2tz)(1-qt^2z)}$.
\end{theorem}

\begin{proof}
Every polyomino $P \in \Para_{2,n}$ is uniquely encoded as a regular expression $ a^*b^*cd^*$ by reading the polyomino from top to bottom and assigning the different rows the following expressions:\\

\centerline{\includegraphics{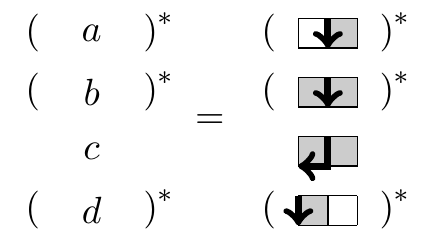}}

In order to count these expressions in terms of the parameters mentioned in the definition of $\sgdistribution$, we weigh the expressions as follows:
$w(a)=qt$ is the weight of a row with only the rightmost cell occupied, 
$w(b)=q^2t$ is the weight of a row with both cells occupied and a vertical step of the bounce path between them, 
$w(c)=q^2t^2$ is the weight of the row where the bounce path changes from vertical to horizontal, 
$w(d)=qt^2$ is the weight of a row with only the leftmost cell and the bounce path on a vertical trajectory towards the origin.
The straightforward translation of the regular expression into a rational generating series gives
$$ \sgdistribution_{2,\star}(q,t;z) = \frac{t^2q^2z}{(1-qtz)(1-q^2tz)(1-qt^2z)}.$$
The involution that exchanges the letters $b$ and $d$, and then exchanges the blocks $b^*$ and $d^*$, proves the symmetry of $\sgdistribution_{2,n}(q,t)$ with respect to $q$ and $t$.  
\end{proof}

\begin{theorem}
$\sgdistribution_{\star, 2}(q,t;z) = {t^2q^2z}/{(1-qtz)(1-q^2tz)(1-qt^2z)}$.
\end{theorem}

\begin{proof}
The case of polyominoes of $\Para_{m,2}$ is similar to $\Para_{2,n}$, although slightly more complicated since it requires a discussion on the height of the first turn in the bounce path.
This discussion leads to the two cases of the column-by-column decomposition of those polyominoes (details left to the reader):\\
\centerline{\includegraphics{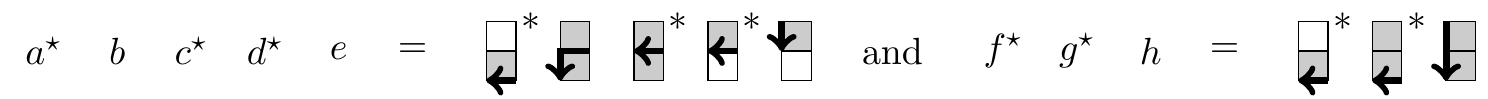}}

The weights of the expressions as follows:
$w(a)=qt^2$, 
$w(b)=q^2 t^3$,
$w(c)=w(g) =q^2 t$, 
$w(d)=w(e)=w(f) = qt$, 
and $w(h)=q^2 t^2$.
The translation of the two regular expressions above into rational series gives
\begin{align*}
\sgdistribution_{\star,2}(q,t;z) 
&= \frac{1}{1-qt^2z}.q^2t^3z.\frac{1}{1-q^2tz}.\frac{1}{1-qtz}.qtz+\frac{1}{1-qtz}.\frac{1}{1-q^2tz}.q^2t^2z \\
&= \dfrac{t^2q^2z}{(1-qtz)(1-q^2tz)(1-qt^2z)}.
\end{align*}
We may also describe explicitly an involution exchanging the two parameters on the words of this regular expression.
The equality of weights 
$$ w(be) = w(b)w(e) = q^2t^3.qt = q^3t^4 = qt^2.q^2t^2 = w(a)w(h) = w(ah)$$ 
shows that 
$$w(a^*bc^*d^*e) = w(a^+c^*d^*h),$$
with a bijection from $a^*bc^*d^*e$ to $a^+c^*d^*h$ that consists of replacing the letter $b$ by $a$ and the letter $e$ by $h$.
In addition, $w(g)=q^2t=w(c)$ and $w(f)=qt=w(d)$ implies that 
$$w(f^*g^*h) = w(c^*d^*h),$$
with a bijection from $f^*g^*h$ to $c^*d^*h$ that consists of replacing the letters $f$ by $d$ and the letters $g$ by $c$, followed by exchanging the order of the blocks of letters $d^*$ and $c^*$.
Subsequently, one has
$$ a^+c^*d^*h + c^*d^*h = a^*c^*d^*h$$
and the resulting language is similar to the case of $\sgdistribution_{2,\star}$, except that now it is the exchange of $c$ and $a$ that defines the involution swapping the two statistics.
\end{proof}

\begin{corollary}
Conjectures~\ref{first:sym} and \ref{second:sym} are true for all pairs $(m,n)$ with $\min(m,n)=2$. 
\end{corollary}

A similar analysis shows that 
$$ \sgdistribution_{3,\star}(q,t;z) = q^3t^3z\left(1+\dfrac{q^3tz}{1-q^3tz}+\dfrac{qt^3z}{1-qt^3z}\right).$$
See Appendix~\ref{more:polys} for higher values.

\newcommand{\atb}{\Upsilon}
\subsection{Proof of Conjecture~\ref{first:sym} for minimal values of one of the two statistics}
\begin{theorem} \label{jeffrey}
In $\Para_{m,n}$, for any $i$, there are as many polyominoes whose bistatistic 
$(\area,\bounceWeight)$ is $(m+n-1,i)$ as there are polyominoes whose bistatistic $(\area,\bounceWeight)$ is $(i,m+n-1)$.  
\end{theorem}

In order to prove this theorem, let us assign names to the two sets of polyominoes with which we are dealing. The
proof will then consist of a bijection $\atb$ that switches the $\area$ and $\bounceWeight$ statistics.
Let 
\begin{align*}
A_{m,n} &= \{\PP \in \Para_{m,n} ~:~ \bounceWeight(\PP)=m+n-1\},\\
B_{m,n} &= \{\PP \in \Para_{m,n} ~:~ \area(\PP)=m+n-1\}.
\end{align*}

\begin{proof}
Let $$\Delta_k = \{ [i-1,i] \times [j-1,j]~:~ i+j-1=k \mbox{ and } i,j \in \mathbb{N}_0\}.$$
For any $\PP \in \Para_{m,n}$ and all $k \in \{1,\ldots,m+n-1\}$, let
$d_k=d_k(\PP)=|\PP \cap \Delta_k|$, the number of cells of $\PP$ on its $k$th anti-diagonal.
Define the vector of all these values $\diagonalLengths(\PP)=(d_1,\ldots d_{m+n-1})$

Suppose that $\PP \in A_{m,n}$.
We now define an operation $\atb$ on $A_{m,n}$.
Polyominoes in $A_{m,n}$ all have the same bounce path that is a sequence of $n$ south steps, followed by a sequence of $m-1$ west steps.
These polyominoes are exactly those which are defined by their upper paths, since their bottom paths are necessarily $\ee^m\nn^n$.

The sequence of diagonal lengths of polyominoes in $A_{m,n}$ describes this upper path: 
for $i \in \{2,\ldots, m\}$, $d_{i-1}<d_i$ (resp. $d_{i-1}=d_i$) represents a north (resp. east) step, 
whereas for $i \in \{m+1,\ldots,  m+n\}$, $d_{i-1}=d_i$ (resp. $d_{i-1}>d_i$) represents a north (resp. east) step.

The position $m$ is special in the diagonal length sequence since it corresponds to the turn of the bottom path: 
indeed $d_1,\ldots d_m$ is weakly increasing while $d_m,\ldots d_{m+n-1}$ is weakly decreasing.
Let $d'=d_m$ and for all $i \in \{0,\ldots,d'\}$ define
\begin{align*}
x_i &= 1+\#\{d_1,\ldots,d_{m-1}>d'-i\}, \\
y_i &= \#\{d_m,\ldots,d_{m+n-1}>d'-i\}.
\end{align*}
Let 
$$\atb(\PP) = \bigcup_{i=1}^{d'} [x_{i-1}-1,x_i -1] \times [y_{i-1},y_{i-1}+1] ~ \cup ~ [x_i -1,x_i] \times [y_{i-1},y_i].$$
By drawing the general diagram of $\atb(\PP)$, one immediately sees that the union above is a union of backwards L shapes that are piled on top of each other, and meet at the endpoints. This means that $\atb(\PP)$ is a ribbon polyomino in $\Ribbon_{m,n}$.
The bounce path is easily seen from the segments it is made up from, 
and we have $$\bounce(\atb(\PP)) = (y_{d'+1-i}-y_{d'-i},x_{d'+1-i}-x_{d'-i},\ldots,y_{1}-y_0,x_1-x_0).$$
The bounce weight is 
\begin{align*}
\bounceWeight(\atb(\PP)) 
&= \sum_{i=1}^{d'} i(y_{d'+1-i}-y_{d'-i} + x_{d'+1-i}-x_{d'-i}) \\
&= \sum_{i=1}^{d'} i(\#\{d_{m},\ldots,d_{m+n-1}=i\} + \#\{d_1,\ldots,d_{m-1}=i\}) \\
&= d_1+\cdots+d_{m+n-1} = \area(\PP).
\end{align*}
So we have that $\atb(\PP) \in B_{m,n}$.
Furthermore, the statistics get swapped:
$$(\area(\atb(\PP)), \bounceWeight(\atb(\PP))) \;=\; (\bounceWeight(\PP),\area(\PP)).$$

The inverse of $\atb$ is straightforward to give:
Suppose the $\PP \in B_{m,n}$ with $\bounce(\PP) = (c_1,c_2,c_4,\ldots,c_{2m+1})$
Then $\atb^{-1}(\PP)$ is the unique polyomino $\PP' \in A_{m,n}$ with 
$$\diagonalLengths(\PP') = (1^{c_2},2^{c_4},\ldots,m^{c_{2m}},(m+1)^{c_{2m+1}},m^{c_{2m-1}},\ldots,1^{c_1}). $$

We omit the remainder of the details.
\end{proof}

\begin{example}
An example of the bijection from Theorem~\ref{jeffrey}. In this case $m=n=6$.
The sequence of diagonal lengths $(d_1,\ldots,d_{m+n-1})=(1,2,3,3,3,\overline{4},3,3,3,2,1)$
give $(x_0,\ldots,x_4)=(1,1,4,5,6)$ and $(y_0,\ldots,y_4)=(0,1,4,5,6)$.
Fill the cells according to $\atb$ to get a ribbon polyomino, $\atb(\PP)$, whose bounce path steps are weighted by the diagonal lengths of $\PP$.
\begin{align*}
\atb(\PP) 
=& [0,0] \times [0,1] ~ \cup ~ [0,1]\times [0,1] ~ \cup ~ [0,3] \times [1,2] ~ \cup ~ [3,4]\times [1,4] \\
& \cup ~ [3,4] \times [4,5] ~ \cup ~ [4,5]\times [4,5] ~ \cup ~ [4,5] \times [5,6] ~ \cup ~ [5,6]\times [5,6].
\end{align*}
\ \\[1em]
\centerline{
\begin{tabular}{c}
\includegraphics{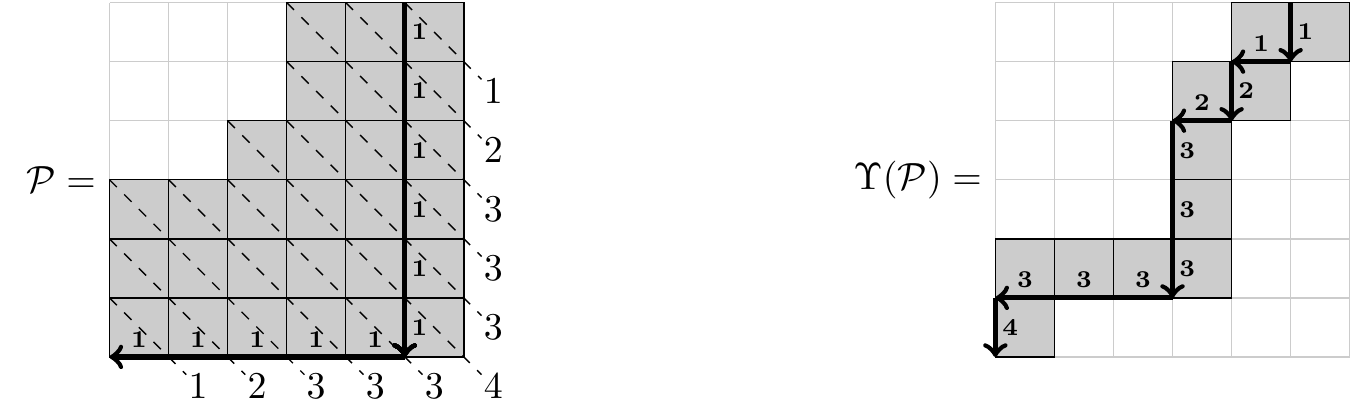}
\end{tabular}
}
\end{example}

This bijection $\atb$ does not seem to extend in a straightforward manner to (at least) all parallelogram polyominoes. 
One reason for this is as follows.
A bounce path of a parallelogram polyomino defines a ribbon polyomino by selecting those cells of the polyomino that are directly above (resp. right) the horizontal (resp. vertical) steps of the bounce path.
In the pictures these are the cells in which we write our labels of the bounce path steps.

Every anti-diagonal (line of slope $-1$) that passes through integer coordinates contains exactly one cell of this bouncing ribbon,
and can therefore be classified as horizontal or vertical based on the step of the bounce path which created that cell.
Our restricted bijection has the `nice' property that, 
in the polyomino of minimal bounce weight, 
a horizontal diagonal of length $i$ is mapped in the polyomino of minimal area 
to an horizontal bounce step of weight $i$, and similarly for vertical steps 
and diagonals.
This property cannot hold in general as is shown by the following example:

\begin{center}
\includegraphics{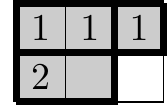}
\end{center}

In this polyomino, both vertical diagonals are of length $1$ and the two 
horizontal diagonals are of lengths 2 and 1.  
If the `nice' property were satisfied then it would imply that in the image 
there is a horizontal bounce step weighted by $2$ without any vertical step 
weighted by 2. From the definition this is impossible. 

\subsection{The sandpile model on $K_n$, special parallelogram polyominoes, and Haglund's bounce statistic}
In this subsection we will show a connection (Theorem~\ref{Olson}) between a class of polyominoes, 
the sandpile model on the complete graph $K_n$ having one sink, and Haglund's bistatistic $(\area,\hbounce)$ on the set of all Dyck paths.
We first need to introduce some notation relevant to Haglund's statistics.

A Dyck path of semi-length $n$ is a path from $(n,n)$ to $(0,0)$ that does not go above the main diagonal and takes steps in $\{\sss=(0,-1), \ww=(-1,0)\}$.
Let $\Dyck_n$ be the set of all such Dyck paths of semi-length $n$.
A general $D \in \Dyck_n$ may be represented as a word $D=d_1 d_2\ldots d_n$ where $d_i \in \{\sss,\ww\}$.
Given $D \in \Dyck_n$, let $\area(D)$ be the number of complete unit squares contained between $D$ and the diagonal line $x=y$. 
(The shaded triangular regions adjacent to the diagonal are not counted.)

Haglund's bounce path for a Dyck path $D$ is the path from $(n,n)$ to the origin, 
starting with an initial south step, and turning from south to west, 
and vice versa every time the path meets the Dyck path $D$ or the line $x=y$. 
It is important to note a subtle difference between Haglund's bounce path and our (polyomino) bounce path.
Haglund's bounce path does {\it{not}} turn when simply hitting a vertex on the lower path (as the polyomino bounce path does).
It needs to hit a west step in order to make a turn to the west. 
The difference is easily seen in the diagram in Example~\ref{HBE}. 
If the path in that diagram had been a polyomino bounce path,
then it would have turned when reaching the vertex (3,2).

If Haglund's bounce path on $D$ is $\sss^{a(1)} \ww^{a(1)} \sss^{a(2)} \ww^{a(2)} \cdots \sss^{a(k)}\ww^{a^(k)}$ then we
will write $\hBounce(D)=(a(1),a(2),\ldots,a(k))$. 
Haglund's bounce statistic, in this setup, is $\hbounce(D) = a(2)+2a(3)+\cdots+(k-1)a(k)$.
Let $$C_n(q,t) = \sum_{D \in \Dyck_n} q^{\area(D)} t^{\hbounce(D)}.$$

\begin{example}\label{HBE}
Let $D=\sss\sss\ww\sss\ww\sss\ww\ww\sss\ww = \sss^2 \ww^1 \sss^1 \ww^1 \sss^1 \ww^2 \sss^1 \ww^1\in \Dyck_5$. Then $\area(D)= 3$, $\hBounce(D)=(2,2,1)$, and $\hbounce=2+2(1)=4$.

\centerline{\includegraphics{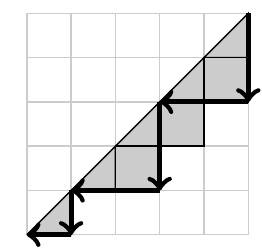}}
\end{example}

Recurrent configurations of the sandpile model on the graph $K_n$ were 
studied in Cori and Rossin~\citep{CoriRossin:2000} and are classified in terms of parking functions.
(See also Knuth~\citep{Knuth:2011}).
Let $D_n$ be the directed graph on the vertices $V=\{v_0,\ldots,v_{n-1}\}$ with single edges $E=\{(v_i,v_j)~:~ i>0 \mbox{ and } 0\leq j \leq n-1\}$.
(In other words vertex $v_0$ is a sink.)
The set of all stable configurations is
$$\Stable(D_n)=\{(x_1,\ldots,x_{n-1})~:~ 0\leq x_i \leq n-2 \mbox{ for all } 1\leq i\leq n-1\}.$$
An integer sequence $(t_1,\ldots,t_n)$ is a \emph{parking function} if there exists a permutation 
$\pi$ of $\{1,\ldots,n\}$ such that $i\geq t_{\pi(i)}$ for all $1\leq i \leq n$.
Cori and Rossin~\cite[Prop. 2.8]{CoriRossin:2000} proved that a stable configuration $(x_1,\ldots,x_{n-1})$ 
is recurrent iff $(n-1-x_1,\ldots,n-1-x_{n-1})$ is a parking function.
Let $\Parking_n$ be the set of all parking functions of length $n$.
The recurrent configurations are
$$\Rec(D_n) = \{x \in \Stable(D_n): (n-1-x_1,\ldots,n-1-x_{n-1}) \in \Parking_{n-1} \}.$$
Let us call a configuration $x \in \Stable(D_n)$ {\emph{sorted}} if it is weakly decreasing, 
and let $\RecDec(D_n)=\{ x\in \Rec(D_n)~:~ x\mbox{ is sorted}\}$.
For example, 
\begin{align*}
\Stable(D_3)&=\{(0,0),(0,1),(1,0),(1,1)\},\\
\Rec(D_3)&=\{(0,1),(1,0),(1,1)\}\\
\RecDec(D_3)&=\{(1,0),(1,1)\}.
\end{align*}
Using Cori and Rossin's~\citep{CoriRossin:2000} classification we have:
$$\RecDec(D_n) = \{(x_1,\ldots,x_{n-1})  ~:~ x_i \geq n-1-i \mbox{ for all } i \in \{1,\ldots,n-1\} \mbox{ and } x_1 \geq \cdots \geq x_{n-1}\}.$$
Section \ref{SEC:3} defined a canonical toppling process for every recurrent configuration.
This canonical toppling process was an ordered set partition of the vertices and recorded the order in which vertices toppled in parallel.
We now extend the same definition to recurrent configurations of $K_n$:
given $x \in \Rec(K_n)$, let $\CanonTop(x) = (Q_1,Q_2,\ldots,Q_k)$ be the ordered set partition of $\{1,\ldots,n-1\}$, whereby vertices in $Q_i$ topples at time $i$.

\begin{example}
Let $x=(4,4,3,2,0) \in \RecDec(K_6)$. The canonical toppling process of this configuration happens as follows:
$$x+(1,1,1,1,1) = (\dot{5},\dot{5},4,3,1) \to (1,1,\dot{6},\dot{5},3) \to (3,3,2,1,\dot{5}) \to (4,4,3,2,0)=x.$$
A dot above the number denotes an unstable vertex that will topple. 
Thus $\CanonTop(x)=(\{1,2\},\{3,4\},\{5\})$.
\end{example}

Given $x \in \Stable(K_n)$, define the diagram of $x$ as 
$$\diag(x)=\bigcup_{i=1}^{n-1} [n-1-i,2+x_i] \times [n-1-i,n-i] \subset \mathbb{R}^2.$$

Diagrams of sorted recurrent configurations in $K_n$ are precisely those diagrams which are polyominoes in $\Para_{n,n-1}$.
Let $$\BSP_{n,n-1} = \{ \diag(x) : x \in \RecDec(K_n)\}.$$
Every element $P \in \BSP_{n,n-1}$ is uniquely described by a Dyck path $\dyck(P)$ of 
semi-length $n-1$ which is the path of the boundary of $P$ from $(n,n-1)$ to $(1,0)$, since the upper path is fixed to $(\nn\ee)^n$.
See Figure~\ref{youngtodyck} for an example of this simple transformation.

\begin{theorem}
Let $u \in \RecDec(K_n)$, $P = \diag(u) \in \BSP_{n,n-1}$ and $D = \dyck(P) \in \Dyck_{n-1}$.
Then the following are equivalent:
\begin{enumerate}
\item[(i)] $\CanonTop(u)=(Q_1,\ldots,Q_k)$,
\item[(ii)] $\Bounce(P)=(|Q_1|,|Q_1|,\ldots,|Q_k|,|Q_k|)$, and
\item[(iii)] $\hBounce(D)=(|Q_1|,|Q_2|,\ldots, |Q_k|)$.
\end{enumerate}
Furthermore:
\begin{enumerate}
\item[(iv)] $\area(P) = (x_1+\cdots+x_n) - (n-1)(n-6)/2 = \area(D)+2(n-1)$.
\end{enumerate}
\end{theorem}

\begin{proof}
The only non-trivial proof is showing the equivalence of (ii) and (iii).
The Haglund bounce path of a Dyck path has a slightly different definition than the bounce path in a parallelogram polyomino. 
However, if one shifts the Dyck path one unit to the right, then the definitions coincide for this particular class thereby giving equivalence.
This equivalence is illustrated in the following diagram:
\begin{center}
\begin{tabular}{c}
\includegraphics{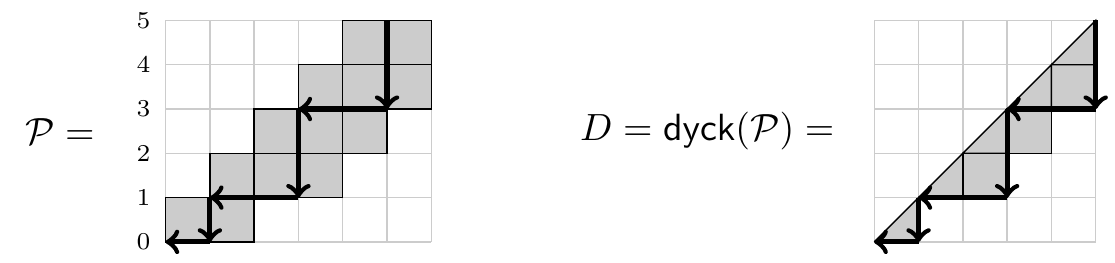}
\end{tabular}
\end{center}
The arrowed path in the diagram on the left shows that $\Bounce(P) = (2,2,2,2,1,1)$. 
The arrowed path in the diagram on the right is Haglund's bounce path (as it is defined on Dyck paths), and so $\hBounce(D)=(2,2,1)$.
\end{proof}

\begin{corollary}
Let $u \in \RecDec(K_n)$, $P = \diag(u) \in \BSP_{n,n-1}$ and $D = \dyck(P) \in \Dyck_{n-1}$.
Then $\bounceWeight(P) = 2(\hbounce(D)+n-1)$.
\end{corollary}

\begin{figure}
\begin{center} \includegraphics{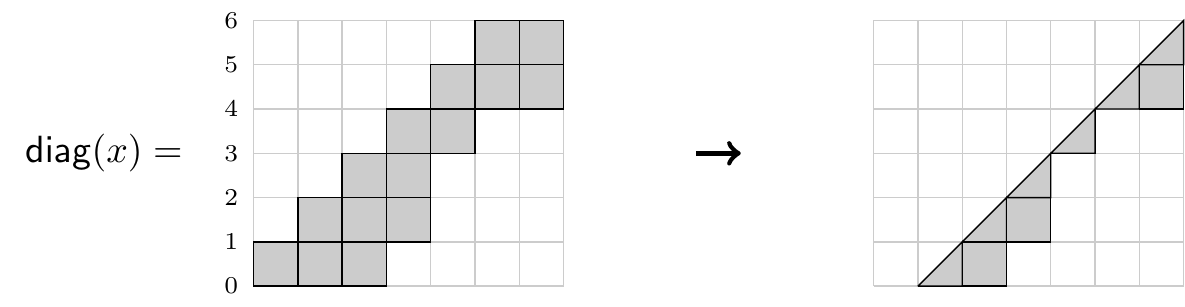} \end{center}
\caption{\label{youngtodyck} On the left is shown $P=\diag(x)$ where $x=(5,5,3,2,2,1)\in \RecDec(K_7)$. 
The Dyck path $\dyck(P)=\sss^2 \ww^2 \sss \ww \sss^2 \ww \sss \ww^2$
is shown on the right.}
\end{figure}

Define the polynomial 
$$\sgdistributionSticked_{n}(q,t) =\displaystyle \sum_{ \PP \in \BSP_{n,n-1}} q^{\area(\PP)} t^{\bounceWeight(\PP)}.$$

\begin{theorem} \label{Olson}
$\sgdistributionSticked_n(q,t) = (qt)^{2(n-1)}\sgdistributionHaglund_{n-1}(q,t^2).$
\end{theorem}

\begin{proof}
We have 
\begin{align*}
\sgdistributionSticked_n(q,t) 
&= \sum_{ \PP \in \BSP_{n,n-1}} q^{\area(\PP)} t^{\bounceWeight(\PP))} \\
&= \sum_{D \in \Dyck_{n-1}} q^{\area(D)+2(n-1)} t^{2(\hbounce(D)+n-1)} \\
&= (qt)^{2(n-1)} \sgdistributionHaglund_{n-1}(q,t^2). \qedhere
\end{align*}
\end{proof}

One is led to Conjecture~\ref{first:sym} by considering two configuration statistics pertaining to (sorted) recurrent configurations on $K_n$.
For the sandpile model on $D_n$, the area of a Dyck path $D$ is related to the number of grains in the mapped sorted recurrent configuration $u$ since
$$ \area(D) + \sum_{i=0}^{n-2} i = \sum_{i=1}^{n-1} u_i.$$
Also, the Haglund bounce statistic of a Dyck path $D$ may be interpreted 
as the sum over all vertices of the number of times a vertex is observed stable before it topples during the parallel toppling process $\CanonTop(u)$.
More precisely, all vertices in $Q_i$ are observed $i$ times as stable: during the initial toppling of the sink and the toppling of each set $Q_1,\ldots Q_{i-1}$.
Conjecture~\ref{first:sym} is the result of translating these configuration statistics to the sandpile model on the graph $D_{m,n}$.

\appendix
\section{Appendix}
\subsection{Two characterisations of parallelogram polyominoes}
The bounce path may be used to give the following characterisation of parallelogram polyominoes.

\begin{characterization}\label{bpcharact}
Let $0\leq a_1 \leq \cdots a_{m-1} <n$ and $0\leq b_1 \leq \cdots \leq b_n<m$ be two weakly increasing sequences of integers. 
Let $P$ be the intersection of the Young diagram whose corner is at $(m,0)$ having heights $(n,1+a_{m-1},\ldots,1+a_1)$ from right to left,
and the Young diagram whose corner is at $(0,n)$ having widths $(1+b_n,\ldots,1+b_1)$ from top to bottom.
Then $P \in \Para_{m,n}$ iff there exist sequences $(x_0,\ldots,x_{k+1})$ and $(y_0,\ldots,y_{k+1})$ with
\begin{enumerate}
\item[(i)] $0=x_0<x_1<\cdots<x_{k}=m-1 <x_{k+1}=m$,
\item[(ii)] $0=y_0\leq y_1<y_2<\cdots<y_{k+1}=n$,
\item[(iii)] $1+a_i \in \{1+y_{j+1},\ldots,y_{j+2}\}$ for all $1+x_j\leq i \leq x_{j+1}$ and $0 \leq j <k$,
\item[(iv)] $1+b_i \in \{1+x_j,\ldots,x_{j+1}\}$ for all $1+y_{j}\leq i \leq y_{j+1}$ and $0\leq j \leq k$.
\end{enumerate}
\end{characterization}

\begin{example} The following polyomino in $\Para_{9,7}$ is formed from the sequences $(a_1,\ldots,a_8)=(2,3,4,4,5,6,6,6)$
and $(b_1,\ldots,b_7)=(0,0,2,4,4,6,8)$. From this one has the sequences: $(x_0,\ldots,x_5)=(0,1,4,5,8,9)$ and $(y_0,\ldots,y_5)=(0,2,3,5,6,7)$ as illustrated in the diagram:\\[1em]
\centerline{\includegraphics{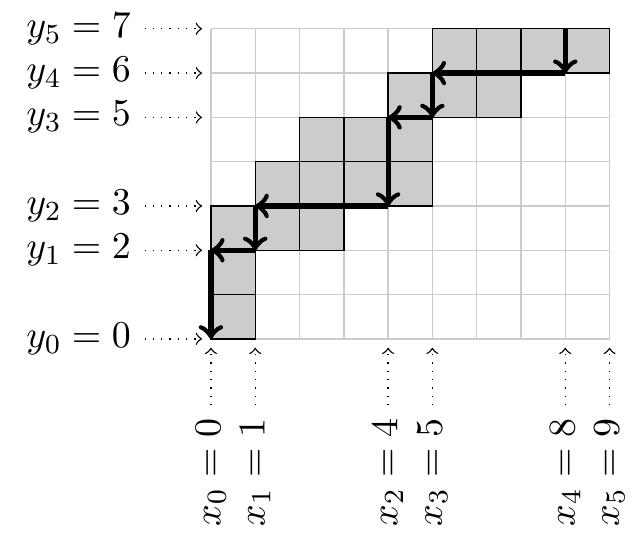}}

The next polyomino is $\PP_2$ from Example 1.2, and is formed from the sequences $(a_1,\ldots,a_6)=(0,0,1,2,2,2)$ and $(b_1,b_2,b_3)=(3,3,6)$.
The $x$ and $y$ sequences are $(x_0,x_1,x_2,x_3)=(0,3,6,7)$ and $(y_0,y_1,y_2)=(0,2,3)$.\\[1em]
\centerline{\includegraphics{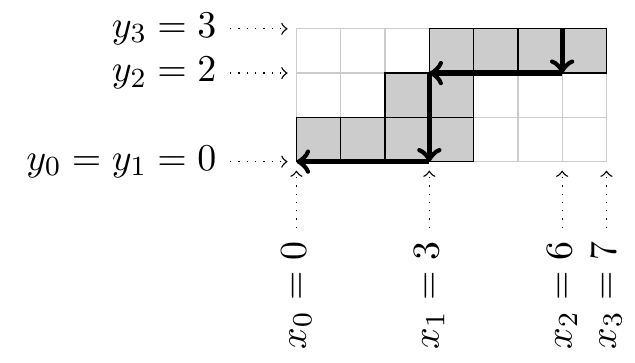}}
\end{example}

We may give an alternative non-graphical characterisation in terms of properties of the integer sequence instead. 

Let $m$ and $n$ be fixed integers.  Given two weakly decreasing sequences $\lambda=(\lambda_1,\ldots,\lambda_m)$ and $\mu=(\mu_1,\ldots,\mu_n)$ let $\Myshape{\lambda}{\mu}{m}{n}$ be the collection of $m \times n$ cells in the plane, with $\lambda_i$ cells missing from the top of column $i$, and $\mu_j$ cells missing from the right of row $j$.  We are interested in pairs of partitions $\lambda$ and $\mu$ which do not `touch' in $\Myshape{\lambda}{\mu}{m}{n}$, as illustrated in the following diagram:\\[1em]
\centerline{\includegraphics{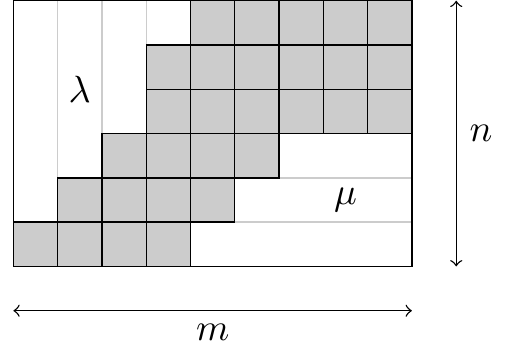}}

\begin{example}
If $\lambda=(3,2,2,1,1,0,0)$ and $\mu=(4,3,1,0)$ then $\Myshape{\lambda}{\mu}{7}{4}$ is $\PP_1$ in Example~\ref{exampleoneone}.
If $\lambda=(2,2,2,1,1,1,0)$ and $\mu=(6,3,0,0)$ then $\Myshape{\lambda}{\mu}{7}{4}$ is $\PP_3$ in Example~\ref{exampleoneone}.
\end{example}

Given $\lambda=(\lambda_1,\ldots,\lambda_m)$ and $\mu=(\mu_1,\ldots,\mu_n)$ contained in the  $m \times n$ rectangle as mentioned above.
Define
\begin{align*}
\lambda_i' &= \# \{ 1\leq k \leq n ~:~ \mu_k > m-i\}\\ 
\mu_j' &= \#\{ 1\leq k \leq m ~:~ \lambda_k > n-j\}.
\end{align*}
for all $i \in \{1,\ldots,m\}$ and $j \in \{1,\ldots,n\}$.
Define $\lambda'=(\lambda_1',\ldots,\lambda_m')$ and $\mu'=(\mu_1',\ldots,\mu_n')$.
Notice that $\lambda_i'$ is the number of missing cells between the bottom of the $m\times n$ bounding rectangle and the lowest cell in column $i$ of $\Myshape{(0,\ldots,0)}{\mu}{m}{n}$.
Similarly $\mu_j'$ is the number of missing cells between the right border of the bounding rectangle and the rightmost cell in row $j$ of $\Myshape{\lambda}{(0,\ldots,0)}{m}{n}$.

\begin{theorem} \label{appendixthm}
Let $\lambda=(\lambda_1,\ldots,\lambda_m)$ and $\mu=(\mu_1,\ldots,\mu_n)$ be two weakly decreasing sequences of non-negative integers where 
$\lambda_i \in \{0,\ldots,n\}$ and $\mu_i \in \{0,\ldots,m\}$.  Then $\Myshape{\lambda}{\mu}{m}{n}\in \Para_{m,n}$ if and only if 
\begin{enumerate}
\item[(i)]  $\mu_{n-\lambda_i}<m-i$ for all $1\leq i < m$,
\item[(ii)] $\lambda_{m-\mu_i}<n-i$ for all $1\leq i < n$, and
\item[(iii)] $\lambda_m = \mu_n = 0$.
\end{enumerate}
\end{theorem}

\begin{proof}
A simple examination of adjacent columns and rows of cells tells us that $\Myshape{\lambda}{\mu}{m}{n} \in \Para_{m,n}$ if and only if
\begin{enumerate}
\item[(i)] $n> \lambda_i +\lambda_{i+1}'$ for all $1\leq i < m$, 
\item[(ii)] $m > \mu_i + \mu_{i+1}'$ for all $1\leq i <n$, and 
\item[(iii)] $\lambda_m = \mu_n = 0$.
\end{enumerate}
Since $\lambda_i'= \#\{\mu_1,\ldots,\mu_n > m-i\}$ we have that condition (i) is equivalent to 
(i') $\#\{ 1\leq k \leq n :  \mu_k < m-i\} > \lambda_i$ for all $1\leq i < m$. 
Similarly, condition (ii) is equivalent to 
(ii') $\#\{ 1\leq k \leq m : \lambda_k < n-i\} > \mu_i$ for all $1\leq i < n$.

Conditions (i') and (ii') may be further simplified by taking advantage of the fact that the sequences are weakly decreasing.
They are equivalent to: (i'') $\mu_{n-\lambda_i}<m-i$ for all $1\leq i < m$, and (ii'') $\lambda_{m-\mu_i}<n-i$ for all $1\leq i < n$.
\end{proof}

It follows from Theorem~\ref{appendixthm} by setting $\lambda_i=n-(1+a_i)$ for all $i \in \{1,\ldots,m-1\}$ 
and $\mu_i = m-(1+b_i)$ for all $i \in \{1,\ldots,n\}$ that we have an alternative characterisation:

\begin{characterization} \label{seqcharact}
Let $0\leq a_1 \leq \cdots a_{m-1} <n$ and $0\leq b_1 \leq \cdots \leq b_n<m$ be two weakly increasing sequences of integers.
Let $P$ be the intersection of the Young diagram whose corner is at $(m,0)$ having heights $(n,1+a_{m-1},\ldots,1+a_1)$ from right to left,
and the Young diagram whose corner is at $(0,n)$ having widths $(1+b_n,\ldots,1+b_1)$ from top to bottom.
Then $P \in \Para_{m,n}$ iff  
\begin{enumerate}
\item[(i)] $i \leq b_{1+a_i}$ for all $1\leq i \leq m-1$,
\item[(ii)] $i \leq a_{1+b_i}$ for all $1\leq i \leq n-1$, and
\item[(iii)] $b_n=m-1$.
\end{enumerate}
\end{characterization}

\subsection{More $F_{\star,\star}(q,t)$ polynomials}
\label{more:polys}
Some more polynomials calculated for Section \ref{SEC:5}. One can check by hand that they are symmetric in $q$ and $t$.
\begin{align*}
{ F_{4,\star}(q,t;z) } 
&= {\dfrac {z{t}^{4}{q}^{4} \left( 
1 
+{q}^{3}{t}^{2}z+{q}^{2}{t}^{3}z
-{q}^{5}{t}^{4}{z}^{2}-{q}^{4}{t}^{5}{z}^{2}
-{q}^{7}{t}^{7}{z}^{3}
\right) }{ 
\left( 1-qtz \right) 
\left( 1-{q}^{2}{t}^{2}z \right) 
\displaystyle\prod_{(i,j) \in \{(1,2),(1,3),(1,4)\}} \left( 1-q^i t^j z  \right) \left( 1-q^j t^i z  \right)
}
}\\
\noalign{}\\
{F_{5,\star}(q,t;z) }
&={\dfrac { \left( 
\begin{array}{l}
1 
+{q}^{2}{t}^{3}z
+{q}^{3}{t}^{2}z
+{q}^{2}{t}^{4}z
+{q}^{4} {t}^{2}z
+{q}^{3}{t}^{3}z
-{q}^{7}{t}^{9}{z}^{3} \\
-{q}^{9}{t}^{7}{z}^{3}
-{q}^{8}{t}^{8}{z}^{3}
-{q}^{8}{t}^{9}{z}^{3}
-{q}^{9}{t}^{8}{z}^{3}
-{q}^{11}{t}^{11}{z}^{4}
\end{array}
\right) {q}^{5}{t}^{5}z}{ 
\left(1- qtz \right)  
\left( 1-{q}^{2}{t}^{2}z \right) 
\displaystyle\prod_{(i,j) \in \{
(1,2),
(1,3),
(1,4),
(1,5)
\}} (1-q^i t^j z)(1-q^jt^i z)
}}\\
\noalign{}\\
{F_{6,\star}(q,t;z) }
&= {\dfrac { \left( 
\begin{array}{l}
1
+{q}^{3}{t}^{3}z
+{q}^{2}{t}^{4}z
+{q}^{4}{t}^{2}z
+{q}^{2}{t}^{5}z
+{q}^{3}{t}^{4}z
+{q}^{4}{t}^{3}z\\
+{q}^{5}{t}^{2}z
-{q}^{6}{t}^{6}{z}^{2}
-{q}^{7}{t}^{4}{z}^{2}
-2\,{q}^{10}{t}^{9}{z}^{3}
-{q}^{11}{t}^{9}{z}^{3}
+{q}^{14}{t}^{11}{z}^{4}\\
+{q}^{13}{t}^{12}{z}^{4}
+{q}^{17}{t}^{14}{z}^{5}
-{q}^{10}{t}^{10}{z} ^{3}
-{q}^{11}{t}^{8}{z}^{3}
+{q}^{15}{t}^{10}{z}^{4}\\
+{q}^{16}{t}^{14}{z}^{5}
-{q}^{18}{t}^{18}{z}^{6}
+{q}^{10}{t}^{15}{z}^{4}
-2\,{q}^{9}{t}^{10}{z}^{3}
+{q}^{11} {t}^{14}{z}^{4}\\
+{q}^{12}{t}^{13}{z}^{4}
-{q}^{17}{t}^{19}{z}^{6}
+{q}^{14}{t}^{17}{z}^{5}
-{q}^{8}{t}^{11}{z}^{3}
-{q}^{9}{t}^{11}{z}^{3}\\
+{q}^{15}{t}^{16}{z}^{5}
+{q}^{16}{t}^{15}{z}^{5}
-{q}^{19}{t}^{17}{z}^{6}
-{q}^{13}{t}^{13}{z}^{4}
-{q}^{21}{t}^{21}{z}^{7}\\
-{q}^{7}{t}^{10}{z}^{3}
-{q}^{5}{t}^{6}{z}^{2}
-{q}^{5}{t}^{7}{z}^{2}
+2\,{q}^{11}{t}^{12}{z}^{4}
-{q}^{9}{t}^{8}{z}^{3}
-{q}^{7}{t}^{11}{z}^{3}\\
+{q}^{10}{t}^{12}{z}^{4}
+{q}^{8}{t}^{8}{z}^{3}
-{q}^{11}{t}^{7}{z}^{3}
+2\,{q}^{12}{t}^{11}{z}^{4}
-{q}^{10}{t}^{8}{z}^{3}\\
+3\,{q}^{12}{t}^{12}{z}^{4}
+{q}^{12}{t}^{10}{z}^{4}
-{q}^{8}{t}^{10}{z}^{3}
+{q}^{11}{t}^{11}{z}^{4}
+{q}^{14}{t}^{16}{z}^{5}\\
-{q}^{17}{t}^{18}{z}^{6}
+{q}^{10}{t}^{14}{z}^{4}
-{q}^{18}{t}^{17}{z}^{6}
-{q}^{4}{t}^{7}{z}^{2}
-{q}^{8}{t}^{9}{z}^{3}\\
+{q}^{10}{t}^{13}{z}^{4}
+{q}^{13}{t}^{10} {z}^{4}
-{q}^{6}{t}^{11}{z}^{3}
+{q}^{11}{t}^{13}{z}^{4}
-3\,{q}^{9}{t}^{9}{z}^{3}\\
+{q}^{13}{t}^{11}{z}^{4}
-{q}^{19}{t}^{16}{z}^{6}
+{q}^{15}{t}^{15}{z}^{5}
+{q}^{14}{t}^{10}{z}^{4}
-{q}^{6}{t}^{5}{z}^{2}\\
-{q}^{10}{t}^{7}{z}^{3}
-{q}^{11}{t}^{6}{z}^{3}
-{q}^{7}{t}^{5}{z}^{2}
-{q}^{16}{t}^{19}{z}^{6} 
\end{array}
\right) {q}^{6}{t}^{6}z}{
\left( 1-qtz \right)  
\left( 1-{q}^{2}{t}^{2}z \right)  
\displaystyle\prod_{(i,j) \in \left\{(1,2),(1,3),(1,4),\atop (1,5),(1,6),(2,3)\right\}} (1-q^it^j z)(1-q^j t^i z)
}}.
\end{align*}

Here we list the polynomials $F_{m,n}(q,t)$ for small values of $m$ and $n$.
We use $ F_{m,n}(q,t) = (qt)^k(m_{i,j})_{0\leq i,j \leq K}$ to abbreviate
$$ F_{m,n}(q,t) = \sum_{0\leq i,j \leq K} m_{i,j}q^{k+i}t^{k+j}.$$

\newcommand{\nextone}{\noindent\hspace*{0cm}}
\nextone
$
\begin{array}{rclrclrcl}
F_{2,2}(q,t) &=& (qt)^3 \Mat {1&1\\ 1&0}; &
F_{3,3}(q,t) &=&  (qt)^5 \Mat{1&1&2&1&1\\ 1&2&2&1&0\\ 2&2&1&0&0\\ 1&1&0&0&0\\ 1&0&0&0&0}; &
F_{2,3}(q,t) &=&  (qt)^4 \Mat{ 1&1&1\\ 1&1&0\\ 1&0&0} ; \\
\noalign{}\\
F_{2, 4}(q,t) &=& (qt)^6 \Mat{1&1&1&1\\ 1&1&1&0\\ 1&1&0&0\\ 1&0&0&0}; & 
F_{3, 4}(q,t) &=& (qt)^6
\Mat{
1&1&2&2&2&1&1\\ 
1&2&3&3&2&1&0\\ 
2&3&4&2&1&0&0\\ 
2&3&2&1&0&0&0 \\ 
2&2&1&0&0&0&0\\ 
1&1&0&0&0&0&0\\ 
1&0&0&0&0&0&0}; &
F_{2, 5}(q,t) &=& (qt)^6
\Mat{
1&1&1&1&1\\ 
1&1&1&1&0\\ 
1&1&1&0&0\\ 
1&1&0&0&0\\ 
1&0&0&0&0};
\end{array}
$
\ \\[1em]

\nextone
$
\begin{array}{rclrclrcl}
F_{4, 4}(q,t) &=& (qt)^7
\Mat{
1&1&2&3&3&3&3&2&1&1\\ 
1&2&4&5&6&5&4&2&1&0\\ 
2&4&7&8&7&5&2&1&0&0\\ 
3&5&8&7&5&2&1&0&0&0\\ 
3&6&7&5&2&1&0&0&0&0\\ 
3&5&5&2&1&0&0&0&0&0\\ 
3&4&2&1&0&0&0&0&0&0\\ 
2&2&1&0&0&0&0&0&0&0\\ 
1&1&0&0&0&0&0&0&0&0\\ 
1&0&0&0&0&0&0&0&0&0};&
F_{3, 5}(q,t) &=& (qt)^7
\Mat{
1&1&2&2&3&2&2&1&1\\ 
1&2&3&4&4&3&2&1&0\\ 
2&3&5&5&4&2&1&0&0\\ 
2&4&5&4&2&1&0&0&0\\ 
3&4&4&2&1&0&0&0&0\\ 
2&3&2&1&0&0&0&0&0\\ 
2&2&1&0&0&0&0&0&0\\ 
1&1&0&0&0&0&0&0&0 \\ 
1&0&0&0&0&0&0&0&0 };\\
\noalign{} \\
F_{2, 6}(q,t) &=& (qt)^7
\Mat{1&1&1&1&1&1\\ 
1&1&1&1&1&0\\ 
1&1&1&1&0&0\\ 
1&1&1&0&0&0\\ 
1&1&0&0&0&0\\ 
1&0&0&0&0&0}.
\end{array}
$

\section*{Acknowledgments}
MD was supported by grant no. 090038013 from the Icelandic Research Fund.
YLB would like to thank the Laboratoire d'informatique Gaspard-Monge, Marne-la-Vall\'ee for their kind hospitality where part of this work was carried out.
The authors would also like to thank the organisers of the `Statistical physics, combinatorics and probability: from discrete to continuous models' trimester at the Centre \'Emile Borel, Institut Henri Poincar\'e, where part of this work was carried out.

\end{document}